\documentclass[11pt]{amsart}
\usepackage{amsmath,amssymb, amsthm,graphicx, url, verbatim, enumerate, relsize, float, amsfonts, amsbsy}
\usepackage{stackrel}
\usepackage[margin=1in]{geometry}
\usepackage{tikz}
\usepackage{color}
\theoremstyle{definition}
\newtheorem{defn}{Definition}[section]

\theoremstyle{plain}
\newtheorem{thm}[defn]{Theorem}

\newtheorem{lem}[defn]{Lemma} 
\newtheorem{cor}[defn]{Corollary}

\newcommand{\Sk}{\mathcal{S}}
\newcommand{\jwproj}{\vcenter{\hbox{\includegraphics[scale=.1]{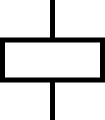}}}}

\newcommand{\sgn}{\text{sgn}}

\graphicspath{{Pics/}}

\begin{document}
\title{Stability properties of the colored Jones polynomial}
\author[C. Lee]{Christine Ruey Shan Lee}
\address[]{Department of Mathematics and Statistics, University of South Alabama, Mobile AL 36608}
\email[]{christine.rs.lee@gmail.com}
\thanks{{Lee was partially supported by NSF grant MSPRF-DMS 1502860.}}

\begin{abstract} It is known that the colored Jones polynomial of a $+$-adequate link has a well-defined tail consisting of stable  coefficients, and that the coefficients of the tail carry geometric and topological information on the $+$-adequate link complement. We show that a power series similar to the tail of the colored Jones polynomial for $+$-adequate links can be defined for \emph{all} links, and that it is trivial if and only if the link is non $+$-adequate.   
\end{abstract}

\maketitle

\section{Introduction} 
For an oriented link $K\subset S^3$, the \emph{colored Jones polynomial} is a sequence of Laurent polynomials $\{J_K^n(q)\}_{n=2}^{\infty}$, where $J_K^n(q) \in \mathbb{Z}[q^{-1/2}, q^{1/2}]$ is an invariant of $K$ for each $n$. In particular, $J_K^2(q)$ is the ordinary Jones polynomial. Unlike the Alexander polynomial, which has a definition based on the topology of the link complement, the colored Jones polynomial originated from representation theory and the ideas of quantum physics, with a definition that makes no explicit reference to the topology of the link. Thus, understanding the geometric and topological information that the polynomial carries  has been a fundamental goal in the study of knots, 3-manifolds, and quantum invariants.

The colored Jones polynomial of a link can be defined and studied in terms of combinatorial properties of link diagrams. One approach to this is through the Kauffman bracket \cite{kauffman}. Studying the combinatorics of the Kauffman bracket, Lickorish and Thistlethwaite \cite{lick-thistle} showed that the extreme degrees of the Jones polynomial for a link are bounded by concrete data from any diagram. If the link is \emph{semi-adequate}, which means that it is \emph{$+$-adequate} or \emph{$-$-adequate}, see Definition \ref{defn:adequate-diagram} and ~\ref{defn:adequate}, then the bounds are sharp. The colored Jones polynomial of a semi-adequate link has since been studied considerably \cite{DL06, stoimenow:coeffs, Arm13, garoufalidisLe, AD13, garouvong1, garouvong, codyoliver}, and it has been shown to relate to the topology of essential surfaces and the geometry of the hyperbolic link complement \cite{Kas97, MM01, MM02, FGL02, q-holo, DL07, fkp:filling, garoufalidis, fkp:gutsjp, fkp:qsf}. 

The coefficients of the colored Jones polynomial exhibit a stability behavior for semi-adequate links. Let $d(n)$ denote the minimum degree of $J_K^n(q)$. For $i \geq 2$,  let $\beta_i$ be the coefficient of $q^{d(i) + (i-2)}$ of $J_K^i(q)$. Armond \cite{Arm13} and Garoufalidis and Le \cite{garoufalidisLe} have independently shown that for all $n \geq i$, the coefficient of $q^{d(n) + (i-2)}$ of $J_K^n(q)$ is equal to $\beta_i$ if $K$ is $+$-adequate. That is, the first $n-1$ coefficients of $J_K^n(q)$ from $d(n)$ are \emph{stable} for a link with a $+$-adequate diagram $D$, extending the results by Dasbach and Lin \cite{DL06} on the first, second, and third coefficient of the colored Jones polynomial of adequate links. For a $+$-adequate link $K$, Armond and Garoufalidis and Le define a power series, called the \emph{tail} of the colored Jones polynomial
\[ T_K(q) = \sum_{i=2}^{\infty} \beta_i q^i,\]  consisting of stable coefficients. These stable coefficients have been shown by the work of Futer, Kalfagianni, and Purcell to carry information on the geometric structure of the complement of a $+$-adequate knot, and to provide sharp volume bounds on the complement of hyperbolic adequate links. See \cite{fkp:gutsjp} for the results and a detailed survey.

Rozansky has shown that stability behavior also occurs in the categorification of the colored Jones polynomial \cite{FKS06, CK12, Roz10}, where a bi-graded chain complex
$C^{Kh}(K, n)$ is assigned to a link $K$, such that the graded Euler characteristic of the homology groups $H^{Kh}(K,n)$ is the \emph{$n$th colored Jones polynomial} $J_K^n(q)$. In \cite{Roz12}, he shows that for $K$ a $+$-adequate link, one can define a family of maps 
\[ f_n: \widetilde{H}^{Kh}(K, n) \rightarrow \widetilde{H}^{Kh}(K, n+1), \] 
such that  $f_n$ is an isomorphism on $\widetilde{H}_{i, *}^{Kh}$ for $i \leq n-1$, where $i$ is the homological grading. The tilde indicates the appropriate degree shift on the chain complex $C^{Kh}(K, n)$, so that $f_n$ is a degree-preserving map for each $n$. In other words, for all $n \geq i$, homology groups $\widetilde{H}^{Kh}(K, n)$  of homological grading less than $i-1$ are isomorphic to the homology groups $\widetilde{H}^{Kh}(K, i)$ of homological grading less than $i-1$. As a result, he defines a \emph{tail homology} of the colored Jones polynomial which is the direct limit of the direct system of $\{ \widetilde{H}^{Kh}(K,n) \}_{n=2}^{\infty}$ determined by $\{f_n\}$. The tail homology contains the \emph{stable} homology groups of $C^{Kh}(K, n)$ as $n$ increases. Therefore, we have a version of the results by Armond \cite{Arm13} and Garoufalidis and Le \cite{garoufalidisLe} for the categorification of the colored Jones polynomial.

It is natural to ask whether similar results can be obtained outside the class of semi-adequate links. In particular, we consider $h_n(D)$, a lower bound for $d(n)$ from any diagram $D$ of the link, see \eqref{eq:lowerbound}. For a $+$-adequate link, we have that $d(n) = h_n(D)$, where $D$ is a $+$-adequate diagram \cite{lick-thistle, fkp:slope}. See Corollary \ref{cor:arealize} for a proof of this result. If a link does not admit a $+$-adequate diagram, we may ask how $d(n)$ behaves with respect to $h_n(D)$. In this paper, we study the effect that a diagram $D$ being \emph{non} $+$-adequate has on the difference $d(n) - h_n(D)$. We work with the unreduced version of the polynomial with the normalization convention as specified by Definition \ref{defn:cjp}. Kalfagianni and the author have shown that $d(n) > h_n(D)$ for $n >2$ if $D$ is not $+$-adequate \cite{kalee1}. Therefore, the colored Jones polynomial can characterize $+$-adequate links. Our main result, which extends that of \cite{kalee1}, is the following.
 
\begin{thm} \label{thm:tail} Let $D$ be a diagram  of a link $K\subset S^3$ and for $n\geq 2$, let $d(n)$ be the minimum degree in $q$ of the $n$th colored Jones polynomial $J_K^n(q)$, with $h_n(D)$ the lower bound of $d(n)$ from $D$ as defined by \eqref{eq:lowerbound}. If $D$ is not $+$-adequate, then $d(n) \geq h_n(D) + n-2$ for $n > 2$.
\end{thm} 

If $K$ is not $+$-adequate, then any diagram $D$ of $K$ is not $+$-adequate. Let $h_n$ be the maximum of $h_n(D)$ taken over all diagrams $D$ of $K$. This is a link invariant and by Theorem ~\ref{thm:tail}, $d(n) \geq h_n + n-2$ for $n>2$. We use this to obtain a link invariant in the form of a power series $J_K^+(q)$, which vanishes if $K$ is not $+$-adequate as in \cite{kalee1}. This allows us to extend the construction of a tail of a $+$-adequate link to all links. See the discussion in Section ~\ref{sec:invariants}. We remark that while $J_K^+(q)$ coincides with $T_K(q)$ when $K$ is adequate, it is not the same as the ``tail" considered in \cite{DL06, Arm13}, since they consider coefficients from the actual minimum degree $d(n)$ of the polynomial. 

In terms of the characterization of semi-adequacy by the colored Jones polynomial, note that Manchon \cite{manchon} has constructed an infinite family of non $+$-adequate knots with diagrams $D$ for which $d(2) = h_2(D)$. An example of such a knot is 12n706, see \emph{KnotInfo} \cite{knotinfo} for a diagram of the knot where $d(2) = h_2(D) = -4$. However, this knot is not $+$-adequate since the first coefficient of its Jones polynomial is 2, and it is known that for a $+$-adequate link, the first coefficient of its Jones polynomial is $\pm 1$ \cite{DL06}. Manchon's examples show that  the degree of the Jones polynomial is not enough to characterize links which are not $+$-adequate, while Kalfagianni \cite{Kal16} has obtained a characterization of adequate links by the colored Jones polynomial using Theorem \ref{thm:tail}. 

A conjecture made by Rozansky \cite[Conjecture 2.13]{Roz12} stating that the tail homology he has constructed vanishes for non $+$-adequate links has now been proven \cite{Lee17}. However, its direct implication for the colored Jones polynomial is weaker than Theorem \ref{thm:tail}, as it states that $d(n) - h_n(D) \geq f(n)$ where $f(n) = an+b$ with constants $a<1$, $b$.

Since a diagram is $-$-adequate if its mirror image is $+$-adequate, for the rest of the paper we will only deal with $+$-adequacy. For the results discussed for a $+$-adequate link, analogous statements for the maximum degree of the colored Jones polynomial of a $-$-adequate link may be obtained by taking the mirror image $\overline{D}$ of a $+$-adequate diagram $D$, using the fact that $J_{\overline{D}}^n(q) = J_{D}^n(q^{-1})$.

\section{Preliminaries} \label{sec:prelim}

\subsection{Skein theory and the colored Jones polynomial} \label{subsec:skein}
We follow the approach of \cite{Lic97}. Let $F$ be an orientable surface which has a finite (possibly empty) collection of points $P$ specified on $\partial F$ if $\partial F \not= \emptyset$. A link diagram on $F$ consists of finitely many arcs and closed curves on $F$ such that 
\begin{itemize}
\item There are finitely many transverse crossings with an over-strand and an under-strand. 
\item The endpoints of the arcs lie in $P$. 
\end{itemize} 
Two link diagrams on $F$ are isotopic if they differ by a homeomorphism of $F$ isotopic to the identity. The isotopy is required to fix $\partial F$. 

\begin{defn}\label{defn:skein} Let $A$ be a fixed complex number. The \emph{linear skein} $\mathcal{S}(F)$ of $F$ is the vector space of formal linear sums over $\mathbb{C}$ of isotopy classes of link diagrams $D$ in $F$ quotiented by the relations 
\begin{enumerate}[(i)]
\item $D \sqcup \vcenter{\hbox{\includegraphics[scale=.10]{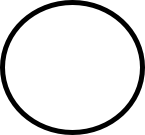}}} = (-A^{-2} - A^{2}) D$, and 
\item $ \vcenter{\hbox{\includegraphics[scale=.2]{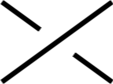}}} = A \ \vcenter{\hbox{\includegraphics[scale=.2]{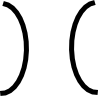}}} \ + A^{-1} \ \vcenter{\hbox{\includegraphics[scale=.2]{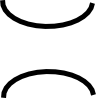}}} \ .$
\end{enumerate} 
\end{defn} 
We consider the linear skein $\mathcal{S}(\mathcal{D}^2, n)$ of the disc $\mathcal{D}^2$ with $2n$-points specified on its boundary, where the boundary is viewed as a rectangle with $n$ marked points above and below, see Figure \ref{fig:TLgen} for an example. For $D_1, D_2 \in \Sk(\mathcal{D}^2,n)$, there is a natural multiplication operation $D_1\cdot D_2$ defined by identifying the top boundary of $D_1$ with the bottom boundary of $D_2$.  This makes $\mathcal{S}(\mathcal{D}^2, n)$ into an algebra $TL_n$, called the $n$th \emph{Temperley-Lieb algebra}. The algebra $TL_n$ is generated by a basis $|_n, e^{1}_n, \ldots, e^{n-1}_{n}$, where $|_n$ is the identity with respect to the multiplication and $e^i_n$ is a crossing-less link diagram as specified below in Figure \ref{fig:TLgen}. 
\begin{figure}[ht]
\def \svgwidth{.5\columnwidth} 
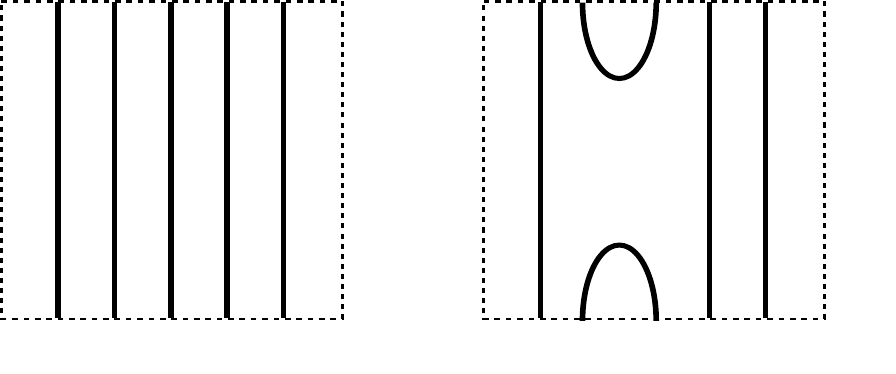
\caption{An example of the identity element $|_n$ and a generator $e^i_n$ of $TL_n$ for $n=5$ and $i=2$. \label{fig:TLgen}}
\end{figure}

We will use a shorthand notation which denotes $n$ parallel strands by $|_n$. 

Suppose that $A^4$ is not a $k$th root of unity for $k\leq n$. There is an element, which we will denote by $\jwproj_n$, in $TL_n$ called the $n$th \emph{Jones-Wenzl projector}, which is uniquely defined  by the following properties. For the original reference where the projector was defined and studied, see \cite{Wen87}. Whenever $n$ is specified we will simply refer to this element as the Jones-Wenzl projector. 
\begin{enumerate}[(i)]
\item $\jwproj_n \cdot e^i_n = e^i_n \cdot \jwproj_n =0$ for $1\leq i \leq n-1$. \label{list:prop1}
\item $\jwproj_n -|_n $ belongs to the algebra generated by $\{e^1_n, e^2_n,\ldots, e^{n-1}_n\}$. 
\item $\jwproj_n \cdot \jwproj_n = \jwproj_n$. 
\item The image of \ $\jwproj_n$ in $\mathcal{S}(\mathbb{R}^2)$, obtained by embedding the disc $\mathcal{D}^2$ in the plane and then joining the $n$ boundary points on the top with those on the bottom with $n$ disjoint planar parallel arcs outside of $\mathcal{D}^2$, is equal to
\[ \frac{(-1)^{n}(A^{2(n+1)}-A^{-2(n+1)})}{A^2-A^{-2}} \cdot \text{the empty diagram in $\mathbb{R}^2$}.\] \label{list:prop4}
\end{enumerate}

To simplify notation, we will let
\[\triangle_n =  \frac{(-1)^{n}(A^{2(n+1)}-A^{-2(n+1)})}{A^2-A^{-2}}.\] 

From the defining properties, the Jones-Wenzl projector also satisfies a recursion relation \eqref{eq:jwrecursive} and other identities \eqref{eq:jwid} and \eqref{eq:jwloopid} as indicated in the following figures. 
\begin{figure}[ht]
\begin{equation} \label{eq:jwrecursive}
\def\svgwidth{.8\columnwidth}
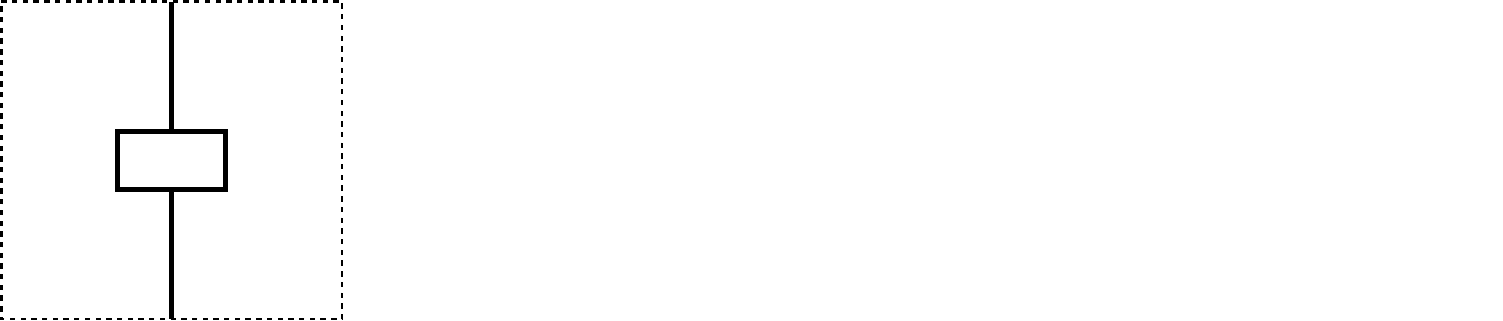
\end{equation}
\caption{A recursive relation for the Jones-Wenzl projector.}
\end{figure} 
\begin{figure}[ht]
\begin{equation} \label{eq:jwid}
\def\svgwidth{.5\columnwidth}
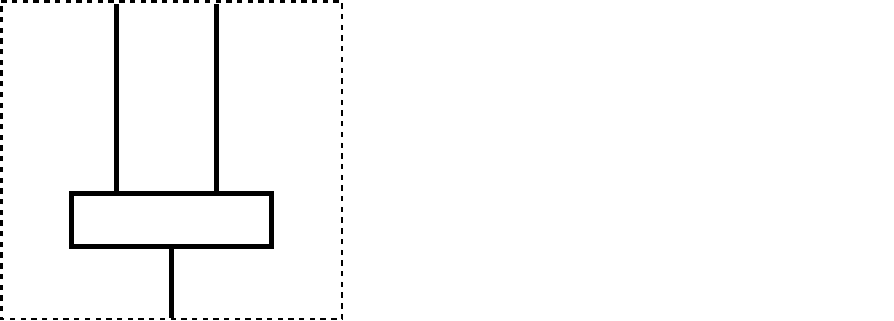
\end{equation}
\caption{An identity for the Jones-Wenzl projector.}
\end{figure} 
\begin{figure}[ht]
\begin{equation} \label{eq:jwloopid}
\def\svgwidth{.5\columnwidth}
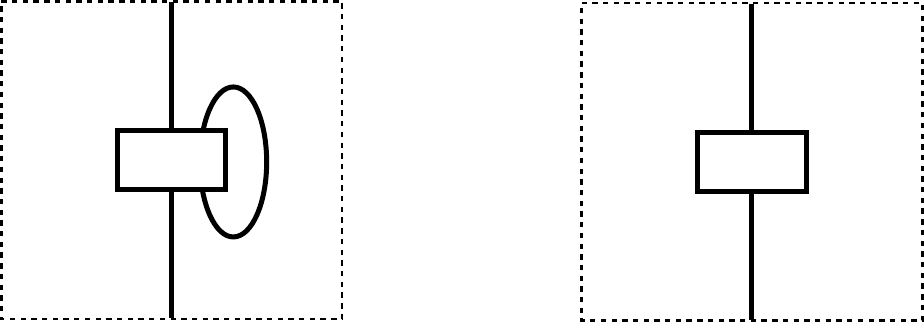
\end{equation}
\caption{Another identity for the Jones-Wenzl projector.}
\end{figure} 

\begin{defn}\label{defn:cjp}
Let $D$ be a diagram of a link $K\subset S^3$ with $k$ components. For each component $D_i$ for $i \in \{1,\ldots, k\}$ take an annulus $A_i$ via the blackboard framing. Let $\mathcal{S}(S^1\times I)$ be the linear skein of the annulus with no points marked on its boundary.
 Let 
\[ f_D: \underbrace{\Sk(S^1\times I) \times \cdots \times \Sk(S^1 \times I)}_{k \text{ times }} \rightarrow \Sk(\mathbb{R}^2)     \] be the map which sends a $k$-tuple of elements $(s_1, \ldots, s_k)$ to $S(\mathbb{R}^2)$ by immersing in the plane the collection of annuli containing the skeins such that the over- and under-crossings of $D$ are the over- and under-crossings of the annuli. 
 For $n\geq 2$, the \emph{$n$th unreduced colored Jones polynomial} $J_K^n(q)$ may be defined as 
\[ J_K^n(q) := ((-1)^{n-1}q^{\frac{(n-1)^2+2(n-1)}{4}})^{\omega(D)} \left\langle f_D\underbrace{\left(\vcenter{\hbox{\def \svgwidth{.05\columnwidth} %% Creator: Inkscape inkscape 0.91, www.inkscape.org
%% Accompanies image file 'jwprojc.pdf' (pdf, eps, ps)
\begingroup%
  \makeatletter%
  \providecommand\color[2][]{%
    \errmessage{(Inkscape) Color is used for the text in Inkscape, but the package 'color.sty' is not loaded}%
    \renewcommand\color[2][]{}%
  }%
  \providecommand\transparent[1]{%
    \errmessage{(Inkscape) Transparency is used (non-zero) for the text in Inkscape, but the package 'transparent.sty' is not loaded}%
    \renewcommand\transparent[1]{}%
  }%
  \providecommand\rotatebox[2]{#2}%
  \ifx\svgwidth\undefined%
    \setlength{\unitlength}{130.11428473bp}%
    \ifx\svgscale\undefined%
      \relax%
    \else%
      \setlength{\unitlength}{\unitlength * \real{\svgscale}}%
    \fi%
  \else%
    \setlength{\unitlength}{\svgwidth}%
  \fi%
  \global\let\svgwidth\undefined%
  \global\let\svgscale\undefined%
  \makeatother%
  \begin{picture}(1,1.09361117)%
    \put(0.17918315,0.01807304){\color[rgb]{0,0,0}\makebox(0,0)[lb]{\smash{$n-1$}}}%
    \put(0,0){\includegraphics[width=\unitlength,page=1]{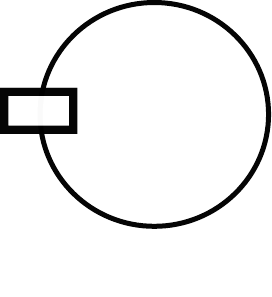}}%
  \end{picture}%
\endgroup%
}},  \vcenter{\hbox{\def \svgwidth{.05\columnwidth} %% Creator: Inkscape inkscape 0.91, www.inkscape.org
%% Accompanies image file 'jwprojc.pdf' (pdf, eps, ps)
\begingroup%
  \makeatletter%
  \providecommand\color[2][]{%
    \errmessage{(Inkscape) Color is used for the text in Inkscape, but the package 'color.sty' is not loaded}%
    \renewcommand\color[2][]{}%
  }%
  \providecommand\transparent[1]{%
    \errmessage{(Inkscape) Transparency is used (non-zero) for the text in Inkscape, but the package 'transparent.sty' is not loaded}%
    \renewcommand\transparent[1]{}%
  }%
  \providecommand\rotatebox[2]{#2}%
  \ifx\svgwidth\undefined%
    \setlength{\unitlength}{130.11428473bp}%
    \ifx\svgscale\undefined%
      \relax%
    \else%
      \setlength{\unitlength}{\unitlength * \real{\svgscale}}%
    \fi%
  \else%
    \setlength{\unitlength}{\svgwidth}%
  \fi%
  \global\let\svgwidth\undefined%
  \global\let\svgscale\undefined%
  \makeatother%
  \begin{picture}(1,1.09361117)%
    \put(0.17918315,0.01807304){\color[rgb]{0,0,0}\makebox(0,0)[lb]{\smash{$n-1$}}}%
    \put(0,0){\includegraphics[width=\unitlength,page=1]{jwprojc.pdf}}%
  \end{picture}%
\endgroup%
}}, \cdots, \vcenter{\hbox{\def \svgwidth{.05\columnwidth} %% Creator: Inkscape inkscape 0.91, www.inkscape.org
%% Accompanies image file 'jwprojc.pdf' (pdf, eps, ps)
\begingroup%
  \makeatletter%
  \providecommand\color[2][]{%
    \errmessage{(Inkscape) Color is used for the text in Inkscape, but the package 'color.sty' is not loaded}%
    \renewcommand\color[2][]{}%
  }%
  \providecommand\transparent[1]{%
    \errmessage{(Inkscape) Transparency is used (non-zero) for the text in Inkscape, but the package 'transparent.sty' is not loaded}%
    \renewcommand\transparent[1]{}%
  }%
  \providecommand\rotatebox[2]{#2}%
  \ifx\svgwidth\undefined%
    \setlength{\unitlength}{130.11428473bp}%
    \ifx\svgscale\undefined%
      \relax%
    \else%
      \setlength{\unitlength}{\unitlength * \real{\svgscale}}%
    \fi%
  \else%
    \setlength{\unitlength}{\svgwidth}%
  \fi%
  \global\let\svgwidth\undefined%
  \global\let\svgscale\undefined%
  \makeatother%
  \begin{picture}(1,1.09361117)%
    \put(0.17918315,0.01807304){\color[rgb]{0,0,0}\makebox(0,0)[lb]{\smash{$n-1$}}}%
    \put(0,0){\includegraphics[width=\unitlength,page=1]{jwprojc.pdf}}%
  \end{picture}%
\endgroup%
}} \hspace{.3cm} \right)}_{k \text{ times }} \right\rangle \biggr\rvert_{q^{1/4} = A^{-1}}, \] 
\end{defn} 
where $\langle D \rangle$ for a linear skein in $\mathcal{S}(\mathbb{R}^2)$ is the polynomial in $A$ multiplying the empty diagram after resolving crossings and removing disjoint circles of $D$ using the skein relations of Definition \ref{defn:skein}. Note that this gives $J_{\vcenter{\hbox{\includegraphics[scale=.07]{circ.png}}}}^n(q) = \triangle_{n-1} \rvert_{q^{1/4}=A^{-1}}$ as the normalization for the value of the colored Jones polynomial of the unknot. To simplify notation, we will write 
\[ D^{n-1}_{\jwproj}=f_D\left(\vcenter{\hbox{\def \svgwidth{.05\columnwidth} %% Creator: Inkscape inkscape 0.91, www.inkscape.org
%% Accompanies image file 'jwprojc.pdf' (pdf, eps, ps)
\begingroup%
  \makeatletter%
  \providecommand\color[2][]{%
    \errmessage{(Inkscape) Color is used for the text in Inkscape, but the package 'color.sty' is not loaded}%
    \renewcommand\color[2][]{}%
  }%
  \providecommand\transparent[1]{%
    \errmessage{(Inkscape) Transparency is used (non-zero) for the text in Inkscape, but the package 'transparent.sty' is not loaded}%
    \renewcommand\transparent[1]{}%
  }%
  \providecommand\rotatebox[2]{#2}%
  \ifx\svgwidth\undefined%
    \setlength{\unitlength}{130.11428473bp}%
    \ifx\svgscale\undefined%
      \relax%
    \else%
      \setlength{\unitlength}{\unitlength * \real{\svgscale}}%
    \fi%
  \else%
    \setlength{\unitlength}{\svgwidth}%
  \fi%
  \global\let\svgwidth\undefined%
  \global\let\svgscale\undefined%
  \makeatother%
  \begin{picture}(1,1.09361117)%
    \put(0.17918315,0.01807304){\color[rgb]{0,0,0}\makebox(0,0)[lb]{\smash{$n-1$}}}%
    \put(0,0){\includegraphics[width=\unitlength,page=1]{jwprojc.pdf}}%
  \end{picture}%
\endgroup%
}},  \vcenter{\hbox{\def \svgwidth{.05\columnwidth} %% Creator: Inkscape inkscape 0.91, www.inkscape.org
%% Accompanies image file 'jwprojc.pdf' (pdf, eps, ps)
\begingroup%
  \makeatletter%
  \providecommand\color[2][]{%
    \errmessage{(Inkscape) Color is used for the text in Inkscape, but the package 'color.sty' is not loaded}%
    \renewcommand\color[2][]{}%
  }%
  \providecommand\transparent[1]{%
    \errmessage{(Inkscape) Transparency is used (non-zero) for the text in Inkscape, but the package 'transparent.sty' is not loaded}%
    \renewcommand\transparent[1]{}%
  }%
  \providecommand\rotatebox[2]{#2}%
  \ifx\svgwidth\undefined%
    \setlength{\unitlength}{130.11428473bp}%
    \ifx\svgscale\undefined%
      \relax%
    \else%
      \setlength{\unitlength}{\unitlength * \real{\svgscale}}%
    \fi%
  \else%
    \setlength{\unitlength}{\svgwidth}%
  \fi%
  \global\let\svgwidth\undefined%
  \global\let\svgscale\undefined%
  \makeatother%
  \begin{picture}(1,1.09361117)%
    \put(0.17918315,0.01807304){\color[rgb]{0,0,0}\makebox(0,0)[lb]{\smash{$n-1$}}}%
    \put(0,0){\includegraphics[width=\unitlength,page=1]{jwprojc.pdf}}%
  \end{picture}%
\endgroup%
}}, \cdots, \vcenter{\hbox{\def \svgwidth{.05\columnwidth} %% Creator: Inkscape inkscape 0.91, www.inkscape.org
%% Accompanies image file 'jwprojc.pdf' (pdf, eps, ps)
\begingroup%
  \makeatletter%
  \providecommand\color[2][]{%
    \errmessage{(Inkscape) Color is used for the text in Inkscape, but the package 'color.sty' is not loaded}%
    \renewcommand\color[2][]{}%
  }%
  \providecommand\transparent[1]{%
    \errmessage{(Inkscape) Transparency is used (non-zero) for the text in Inkscape, but the package 'transparent.sty' is not loaded}%
    \renewcommand\transparent[1]{}%
  }%
  \providecommand\rotatebox[2]{#2}%
  \ifx\svgwidth\undefined%
    \setlength{\unitlength}{130.11428473bp}%
    \ifx\svgscale\undefined%
      \relax%
    \else%
      \setlength{\unitlength}{\unitlength * \real{\svgscale}}%
    \fi%
  \else%
    \setlength{\unitlength}{\svgwidth}%
  \fi%
  \global\let\svgwidth\undefined%
  \global\let\svgscale\undefined%
  \makeatother%
  \begin{picture}(1,1.09361117)%
    \put(0.17918315,0.01807304){\color[rgb]{0,0,0}\makebox(0,0)[lb]{\smash{$n-1$}}}%
    \put(0,0){\includegraphics[width=\unitlength,page=1]{jwprojc.pdf}}%
  \end{picture}%
\endgroup%
}} \hspace{.3cm} \right)\] for the rest of this paper.

\subsection{Semi-adequate links} \label{subsec:semi-adequate}
Let $D$ be a diagram of a link $K$ in $S^3$. A \emph{Kauffman state} is a choice of replacing every crossing of $D$ by the $+$- or $-$-resolution as in Figure \ref{fig:abres}, with the dashed segment recording the location of the crossing before the replacement.

\begin{figure}[ht]
\centering
\def \svgwidth{.8\columnwidth}
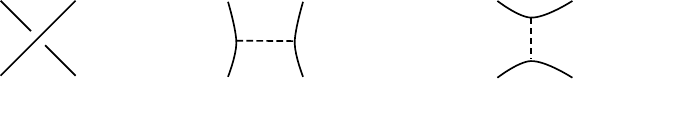
\caption{$+$- and $-$-resolutions of a crossing.}
\label{fig:abres}
\end{figure} 

Applying a Kauffman state results in a set of disjoint circles called \emph{state circles}. We form a $\sigma$-\emph{state graph} $s_{\sigma}(D)$ for each Kauffman state $\sigma$ by letting the resulting state circles be vertices and the segments be edges. The \emph{all-$+$} state graph $s_+(D)$ comes from the Kauffman state which chooses the $+$-resolution at every crossing of $D$. Similarly, the \emph{all-$-$} state graph $s_-(D)$ comes from the Kauffman state which chooses the $-$-resolution at every crossing of $D$.

\begin{defn}\label{defn:adequate-diagram} A link diagram $D$ is \emph{$+$-adequate} if its all-$+$ state graph has no one-edged loops (a one-edged loop is an edge that has both ends on the same vertex). A diagram is \emph{$-$-adequate} if its mirror image is $+$-adequate.  
\end{defn} 

\begin{defn} \label{defn:adequate} A link $K$ is \emph{semi-adequate} (\emph{$+$- or $-$-adequate}) if it admits a diagram that is $+$- or $-$-adequate. 
\end{defn} 

We consider the following combinatorial data of an oriented link diagram $D$ and a Kauffman state $\sigma$ on $D$.
\begin{itemize}
\item $c(D):=$ The number of crossings in $D$.
\item $\omega(D):=$ The writhe of $D$. 
\item $|s_{\sigma}(D)|:=$ The number of vertices in the $\sigma$-state graph of $D$. 
\item $\sgn(\sigma) := \sgn_+(\sigma) - \sgn_-(\sigma),$ where
\begin{align} 
\sgn_+(\sigma) &:= \text{The number of crossings where the $+$-resolution is chosen in $\sigma$, and}  \\  \notag
\sgn_-(\sigma) &:= \text{The number of crossings where the $-$-resolution is chosen in $\sigma$}.\end{align}
\end{itemize} 
We will use the special notations $``\sgn(+)"$ and $``\sgn(-)"$ for the sign $\sgn$ of the all-$+$ state,  and the sign of the all-$-$ state, respectively. 
Let
\begin{align}
H_n(D) &= n^2c(D) + 2n|s_+(D)| \label{eq:upperbound} 
\intertext{and}
h_n(D) &= -\frac{1}{4}H_{n-1}(D)
 + \frac{(n-1)^2+2(n-1)}{4}\omega(D). \label{eq:lowerbound} \end{align}

We use the following lemmas from skein theory involving the notion of adequacy. The degree of a rational function that we consider is the maximum power of the Laurent series expression of the function with the power of each term in the series bounded from above.

\begin{defn} Let $\Sk$ be a skein in $\Sk(\mathbb{R}^2)$ with crossings. Let $\sigma$ and $\sigma'$ be Kauffman states on the crossings of $\Sk$. We say that there is a \emph{sequence of states} from $\sigma$ to $\sigma'$ if there is a finite sequence $\sigma_1=\sigma, \sigma_2, \ldots, \sigma_f = \sigma'$ such that for each $i$, there is a distinct crossing $x_i$ on which $\sigma_i$ chooses the $A$-resolution and $\sigma_{i+1}$ chooses the $B$-resolution, and they are otherwise the same on every other crossing of $\Sk$.
\end{defn} 
Note that there is a sequence of states from the all-$+$ state to any other state. 

\begin{defn} Let $\Sk$ be a skein in $\Sk(\mathbb{R}^2)$ which may have crossings and  may be decorated by Jones-Wenzl projectors, we will denote by $\overline{\Sk}$ the skein obtained from $\Sk$ where every projector is replaced by the identity. 
\end{defn} 
\begin{lem}{\cite[Lemma 5.6]{Lic97}} \label{lem:sstates} Let $\sigma$ and $\sigma'$ be two Kauffman states on a skein $\Sk \in \Sk(\mathbb{R}^2)$ with a sequence of states $\{ \sigma_i \}_{i=1}^f$ from $\sigma$ to $\sigma'$, and let $\overline{\Sk_{\sigma}}$ and $\overline{\Sk_{\sigma'}}$ be the skeins resulting from applying the Kauffman states $\sigma$ and $\sigma'$, respectively, to $\Sk$ and replacing every Jones-Wenzl projector by the identity. Suppose $m\geq 0$ is the number of pairs $(\sigma_i, \sigma_{i+1})$ in the sequence where the number of circles in $\overline{\Sk_{\sigma_{i+1}}}$ is one fewer than that of $\overline{\Sk_{\sigma_i}}$, then  
\[\deg \left( A^{\sgn(\sigma')}\langle \overline{\Sk_{\sigma'}} \rangle\right) \leq \deg \left(A^{\sgn(\sigma)}\langle \overline{\Sk_{\sigma}} \rangle\right) - 4m. \] 
\end{lem} 
Note that from $\sigma_i$ to $\sigma_{i+1}$ in a sequence, either a pair of circles is merged or a circle is split into two. 

The definition of an adequate skein is due to Armond \cite{Arm13}. 
\begin{defn} Let $\Sk \in \Sk(\mathbb{R}^2)$ be a crossing-less skein decorated by Jones-Wenzl projectors $\jwproj_n$. Consider the skein $\overline{\Sk}$ constructed from $\Sk$ by replacing each of the Jones-Wenzl projectors by the identity in $TL_n$. Consider the regions in $\overline{\Sk}$ where the projectors had previously been. We say that $\Sk$ is \emph{adequate} if no circle in $\overline{\Sk}$ passes through any of these regions more than once. 
\end{defn}

\begin{lem}[{\cite[Lemma 4]{Arm13}}]\label{lem:jwid} Let $\Sk \in \Sk(\mathbb{R}^2)$ be a crossing-less skein decorated by Jones-Wenzl projectors $\jwproj_n$, and let $\overline{\Sk}$ be the skein obtained by replacing each Jones-Wenzl projector by the identity element $ |_n$, then 
\[\deg  \langle \Sk \rangle  \leq \deg \langle \overline{\Sk} \rangle,    \] and equality is achieved when $\Sk$ is adequate.  
\end{lem}
The next lemma follows immediately from \cite[Lemma 5.6]{Lic97}. 
\begin{lem}\label{lem:Adegree} Let $\Sk\in \Sk(\mathbb{R}^2)$ be a skein with crossings decorated by the Jones-Wenzl projector $ \jwproj_n$ for some fixed $n$, and let $\Sk_+$ be the skein resulting from applying the all-$+$ Kauffman state on the crossings of $\Sk$, then 
\[\deg\langle \Sk \rangle \leq \deg \left( A^{\sgn(+)} \langle \overline{\Sk_+} \rangle \right).   \] 
\end{lem}
\begin{proof}
Let $\Sk_{\sigma}$ be the skein resulting from applying the Kauffman state $\sigma$ on the crossings of $\Sk$. Then 
\[ \langle \Sk \rangle =\sum_{\sigma}  A^{\sgn(\sigma)} \langle \Sk_{\sigma} \rangle .\] By Lemma \ref{lem:jwid}, 
\[ \deg \langle \Sk_{\sigma} \rangle  \leq \deg \langle \overline{\Sk_{\sigma}}\rangle.  \] 
Now 
\[ \deg \left( A^{\sgn(\sigma)} \langle \overline{\Sk_{\sigma}} \rangle \right) \leq \deg \left( A^{sgn(+)} \langle \overline{\Sk_{+}} \rangle \right)\] by considering the link diagram $\overline{\Sk}$ and applying Lemma 2.7, since there is a sequence of states from the all-$+$ state to any other state. 
\end{proof}  

We use Lemma \ref{lem:jwid} and \ref{lem:Adegree} to give a proof of the following fact used to establish the degree of the colored Jones polynomial of a $+$-adequate link \cite{Lic97}.
\begin{cor} \label{cor:arealize} Let $D$ be a link diagram and $D^n_{\jwproj}$ be the $n$-blackboard cable of $D$ with each component decorated by a Jones-Wenzl projector as in Definition \ref{defn:cjp}, then 
\[\deg \langle D^n_{\jwproj} \rangle \leq H_n(D).\] Therefore, 
\[ h_n(D) \leq d(n),\] and equality is achieved when $D$ is $+$-adequate.   
\end{cor}
\begin{proof}
Let $\sigma$ be a Kauffman state on the set of crossings of $D^n_{\jwproj}$, with $S^n_{\sigma}$ the skein resulting from applying the state $\sigma$ and $S^n_+$ the skein resulting from applying the all-$+$ state. The inequality follows from writing 
\[ \langle D^n_{\jwproj} \rangle  = \sum_{\sigma} A^{\sgn(\sigma)}\langle \Sk^n_{\sigma} \rangle, \] 
Lemma \ref{lem:Adegree}, and recognizing that $\deg \left( A^{\sgn(+)} \langle \overline{ S^n_{+}} \rangle \right) = H_n(D)$. The skein $\Sk^n_+$ is adequate, so by Lemma \ref{lem:jwid} we have  
\[ \deg \langle S^n_{+} \rangle = \deg \langle \overline{ S^n_{+}} \rangle.  \]   Now if $D$ is $+$-adequate, then $D^n$ is also $+$-adequate, which one can directly check as in  \cite[Lemma 5.12]{Lic97}. A sequence of states from the all-$+$ state to any other state necessarily contains a pair which merges a pair of circles. This implies
\[ \deg \left( A^{\sgn(\sigma)}\langle \overline{ \Sk^n_{\sigma}} \rangle \right) <  \deg \left( A^{\sgn(+)} \langle  \overline{\Sk^n_{+}}\rangle \right).\]  Thus
\[ \deg \langle D^n_{\jwproj} \rangle = \deg \left( A^{\sgn(+)} \langle  \overline{\Sk^n_{+}}\rangle \right) = H_n(D). \]

\end{proof} 

\section{Proof of Theorem \ref{thm:tail}} \label{sec:main}
Throughout this section, we assume that $D$ is not $+$-adequate, and therefore, it has a one-edged loop in its all-$+$ state graph. Recall that, as in Definition \ref{defn:cjp}, we may obtain $J^n_K(q)$ by evaluating the Kauffman bracket on $D^n_{\jwproj}$, the $n$-blackboard cable $D^n$ decorated by a Jones-Wenzl projector. In our case, we will consider the decoration by four Jones-Wenzl projectors around a fixed cabled crossing $c$, which corresponds to a loop in the all-$+$ state graph of $D$. See Figure \ref{fig:cabledcrossing} for what is meant by cabling a crossing and Figure \ref{fig:dcrossing} for an example of a chosen crossing framed by projectors. Without loss of generality, we assume the loop is attached on the inside of a state circle $S$ of $s_+(D)$. Since this skein is obtained from $D^n_{\jwproj}$ by doubling and sliding the projector on each component using the defining properties of the projector, it is equivalent to $D^n_{\jwproj}$ and we will also denote the resulting skein by $D^n_{\vcenter{\hbox{\includegraphics[scale=.1]{jwproj.png}}}}$ by a slight abuse of notation.  
\begin{figure}[ht]
\begin{center}
\def\svgwidth{.8\columnwidth}
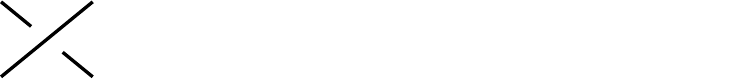
\end{center}
\caption{\label{fig:cabledcrossing} A crossing and its 3-cable.}
\end{figure}
\begin{figure}[ht]
\begin{center}
\def\svgwidth{\columnwidth}
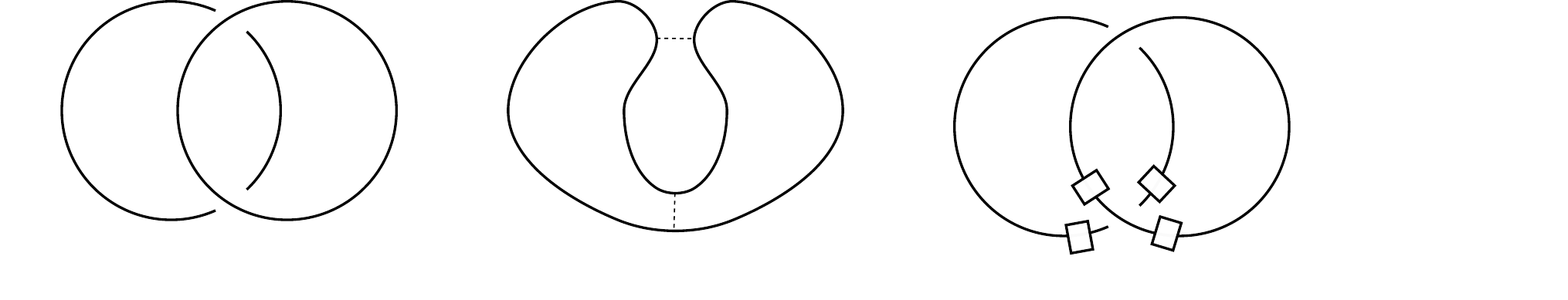
\end{center}
\caption{\label{fig:dcrossing} We put four Jones-Wenzl projectors around a single crossing $c$.}
\end{figure}

Theorem \ref{thm:tail} will follow from the equivalent statement given below for the Kauffman bracket with variable $A$. Recall that by Corollary \ref{cor:arealize}, for any link diagram we have 
\[\deg \langle D^n_{\jwproj} \rangle \leq H_n(D).\]  
\begin{thm} \label{thm:equiv}
Let $K$ be a link with diagram $D$ and $H_n(D)$ be defined from $D$ as in \eqref{eq:upperbound}. If $D$ is not $+$-adequate, then $\deg \langle D^n_{\jwproj} \rangle \leq H_n(D) - 4(n-1)$ for $n\geq 1$. 
\end{thm}
Henceforth, we will work exclusively with the Kauffman bracket. 

Let $\Sk$ be a skein with crossings in $\Sk(\mathbb{R}^2)$ which may or may not be decorated by Jones-Wenzl projectors. We denote by $\Sk_{\sigma}$ the crossing-less skein obtained by applying a Kauffman state $\sigma$ to the crossings of $\Sk$. A $\sigma$-state graph $s_{\sigma}(\Sk)$ is then the set of disjoint circles with segments as before, except for the presence of projectors. 

Let $D$ be a link diagram and consider a skein $\Sk^n_{\sigma}$ obtained by applying a Kauffman state $\sigma$ to the crossings of $D^n_{\vcenter{\hbox{\includegraphics[scale=.1]{jwproj.png}}}}$. We have the following state sum for the Kauffman bracket of $D^n_{\jwproj}$. 
\begin{equation} \langle D^n_{\jwproj} \rangle = \sum_{\sigma} A^{\sgn(\sigma)} \langle \mathcal{S}^n_{\sigma} \rangle. \end{equation}

\subsection*{Strategy for proof of Theorem \ref{thm:equiv}} The strategy for the proof of Theorem \ref{thm:equiv} is to characterize the states in the above equation relevant to the last $n-2$ coefficients of $\langle D^n_{\jwproj} \rangle$ from $H_n(D)$ for $D$ a non $+$-adequate diagram. Then, Lemma \ref{lem:avariant}  and Lemma \ref{lem:localform} are used to relate these states to the skein of a diagram of the unknot in Definition \ref{defn:unknotwist}, which establishes that these last $n-2$ coefficients are trivial by Lemma \ref{lem:generic-loop}.

We make the following definition which generalizes the $\doteq_n$ equivalence in \cite[Pg. 1]{Arm13}. 
\begin{defn}
Let $s$ and $m$ be two integers $\geq 0$ and $P_1(A)$ and $P_2(A)$ be two Laurent series in $A$. We write 
\[P_1(A) \doteq^s_m P_2(A) \] if and only if 
the coefficients of $A^s, A^{s-4(1)}, \ldots, A^{s-4(m-1)}$ and $A$'s with power $> s$ in $P_1(A)$ agree with those of $P_2(A)$. For example, $2A^9-A^5 +A^1 \doteq^9_2 2A^9-A^5 + 4A^1$. For two skeins $\Sk_1$, $\Sk_2 \in \Sk(\mathbb{R}^2)$, we write $\Sk_1 \doteq^s_m \Sk_2$ if $\langle \Sk_1 \rangle \doteq^s_m \langle \Sk_2 \rangle$.
\end{defn}

The following lemma is an important variant of \cite[Lemma 10]{Arm13}.

\begin{lem} \label{lem:avariant}
Let $0\leq k\leq n$.
\begin{figure}[ht]
\fontsize{8}{10}\selectfont
\def \svgwidth{\columnwidth}
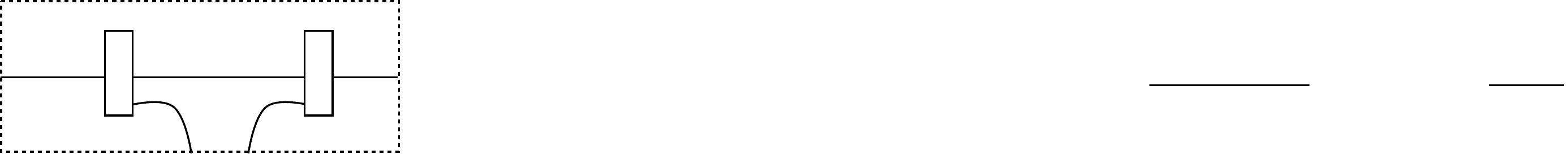
\end{figure} 
\end{lem}

\begin{proof}
Similar to the proof of \cite[Lemma 10]{Arm13}, we apply the recursive relation \eqref{eq:jwrecursive} to the left projector in the picture on the left side of the equality above to get the following equality as indicated in \eqref{eq:avariant2}.
\begin{figure}[H]
\fontsize{8}{10}\selectfont
\def \svgwidth{\columnwidth}
\begin{equation}
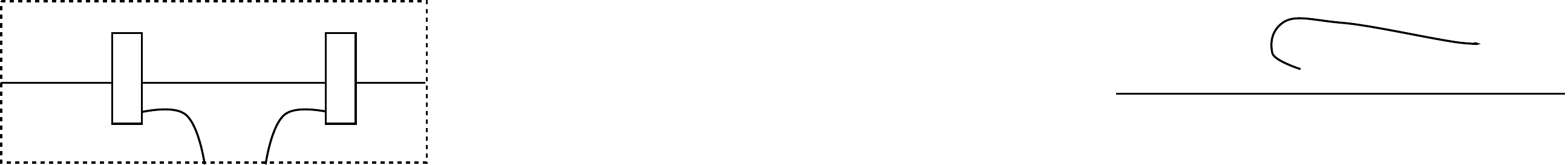 \label{eq:avariant2}
\end{equation}
\end{figure} 

For $n-k>1$, we apply the recursive relation again to the middle projector in the last picture in the above equation.
\begin{figure}[ht]
\fontsize{8}{10}\selectfont
\def \svgwidth{\columnwidth}
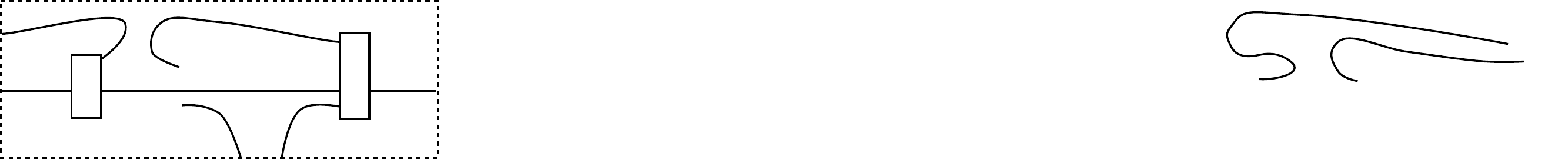
\end{figure} 
It is clear that the second term in the resulting sum is zero, due to property (i) of the Jones-Wenzl projector. 
Sliding to the left the second projector from the left in the last picture above by \eqref{eq:jwid}, and combining with \eqref{eq:avariant2}, we see that we get the lemma by repeated expansion (($n-k$) times) via the recursion relation. 
\end{proof}

We consider a particular type of skeins from a diagram of the unknot and study their $\doteq^s_m$ equivalences using Lemma \ref{lem:avariant}. 
\begin{defn} \label{defn:unknotwist} Let $U^n_{\jwproj}$ be the $n$-blackboard cable of the diagram of the unknot with one left-hand half-twist added and decorated by four Jones-Wenzl projectors $\jwproj_n$ framing the crossing. Consider the set of disjoint circles $\overline{U_+^n}$ from applying the all-$+$ Kauffman state to the crossings of $U^n_{\jwproj}$, and then replacing each Jones-Wenzl projector by the identity. Label each of them, innermost first, as $S_1, \ldots, S_n$. For $1\leq j < n$, let $U^{n, j}$ be the skein obtained from $U_+^n$ by removing all circles and segments resulting from applying the all-$+$ Kauffman state outside of $S_{j}$, and replacing each remaining segment by the corresponding crossing before choosing the $+$-resolution. See Figure \ref{fig:sskein} for an example. \end{defn} 

\begin{figure}[H]
\def \svgwidth{\columnwidth}
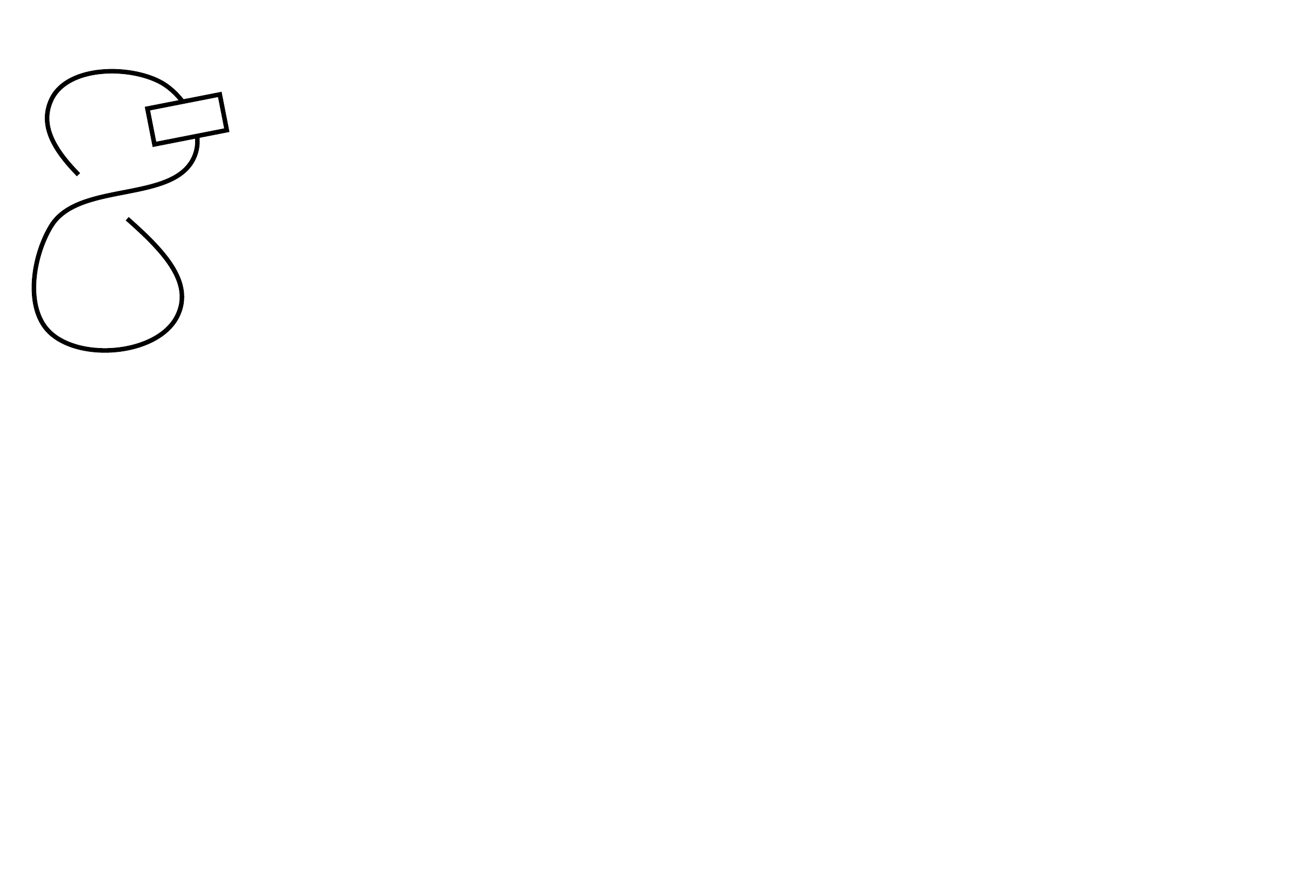
\caption{\label{fig:sskein} An example of $U^{n,j}$ where $j = 3$ and $n = 8 = 2j + 2$.}
\end{figure} 
\ \\ 

The key lemma below computes a degree bound on $\langle U^{n,j}\rangle$ for $U^{n, j}$ from Definition \ref{defn:unknotwist}. 

\begin{lem} \label{lem:generic-loop} Let $U^{n, j}_+$ be the skein obtained from $U^{n, j}$ as defined in Definition \ref{defn:unknotwist} by applying the all-$+$ Kauffman state to its crossings, then we have for $1\leq j < n$, 
\begin{equation} \label{eq:specbound}
\deg \langle U^{n, j} \rangle \leq n^2-(n-j)(n-j+1) + \deg \langle \overline{U^{n, j}_+} \rangle - 4j. \end{equation}
\end{lem}
\begin{proof}
Firstly, the number of crossings of $U^{n, j}$ is $n^2-(n-j)(n-j+1)$, thus 
\[ \deg \langle U^{n, j} \rangle \leq \underbrace{n^2-(n-j)(n-j+1) + \deg \langle \overline{U^{n, j}_+} \rangle}_{\text{general upper bound}} \] by Lemma \ref{lem:Adegree}.  Write $n = n'j + r$ with the smallest nonnegative remainder $r$. We have $n'\geq 1$ because of the assumption that $j< n$. The skein $U^{n, j}$ has $n'$ full kinks on $j$ strands from examining the braid word and removing crossings with the presence of the projectors. Removing each full twist on $j$ strands will decrease the degree by $j^2+2j$ \cite[Lemma 14.1]{Lic97} in addition to removing $j^2$ crossings. Thus the degree decreases from the upper bound $n^2-(n-j)(n-j+1) + \deg \langle \overline{U^{n, j}_+} \rangle$ from taking the all-$+$ state by $n'(2j^2+2j) \geq 4j$ for $n', j \geq 1$. 
\end{proof}

Now we consider skeins which can be reduced to $U^{n, j}$ through $\doteq^{s}_{*}$-equivalences.

\begin{lem} \label{lem:localform} Suppose that we have a skein $\Sk$ in $\Sk(\mathbb{R}^2)$ of the form shown below in Figure \ref{fig:twocase}. Consider the set of circles of $\overline{\Sk}$ and number the circles, innermost first, as $S_1, \ldots, S_j$, where $j = n-(\ell + \ell')$. Let $L^{n, j}$ be the subset of the crossings of $U^{n,j}$ inside $S_j$, and let $|L^{n, j}|$ be the number of crossings in  $L^{n, j}$.

\begin{figure}[H]
\def \svgwidth{.4\columnwidth}
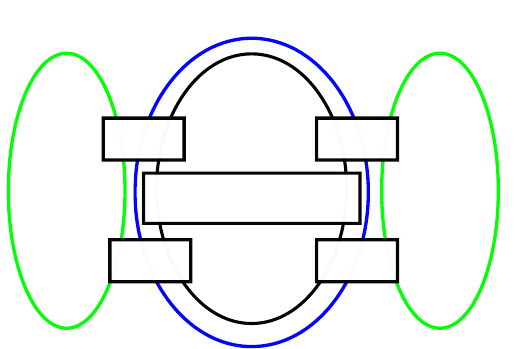
\caption{\label{fig:twocase} The skein $\Sk$ composed of crossings from $U^{n, j}$ and $\ell + 2\ell'$ circles attached to the projectors.}\end{figure}

Then, 
\begin{equation}
\def \svgwidth{.7\columnwidth}
%% Creator: Inkscape inkscape 0.91, www.inkscape.org
%% Accompanies image file 'gformeq.pdf' (pdf, eps, ps)
\begingroup%
  \makeatletter%
  \providecommand\color[2][]{%
    \errmessage{(Inkscape) Color is used for the text in Inkscape, but the package 'color.sty' is not loaded}%
    \renewcommand\color[2][]{}%
  }%
  \providecommand\transparent[1]{%
    \errmessage{(Inkscape) Transparency is used (non-zero) for the text in Inkscape, but the package 'transparent.sty' is not loaded}%
    \renewcommand\transparent[1]{}%
  }%
  \providecommand\rotatebox[2]{#2}%
  \ifx\svgwidth\undefined%
    \setlength{\unitlength}{615.94831332bp}%
    \ifx\svgscale\undefined%
      \relax%
    \else%
      \setlength{\unitlength}{\unitlength * \real{\svgscale}}%
    \fi%
  \else%
    \setlength{\unitlength}{\svgwidth}%
  \fi%
  \global\let\svgwidth\undefined%
  \global\let\svgscale\undefined%
  \makeatother%
  \begin{picture}(1,0.27138362)%
    \put(0,0){\includegraphics[width=\unitlength,page=1]{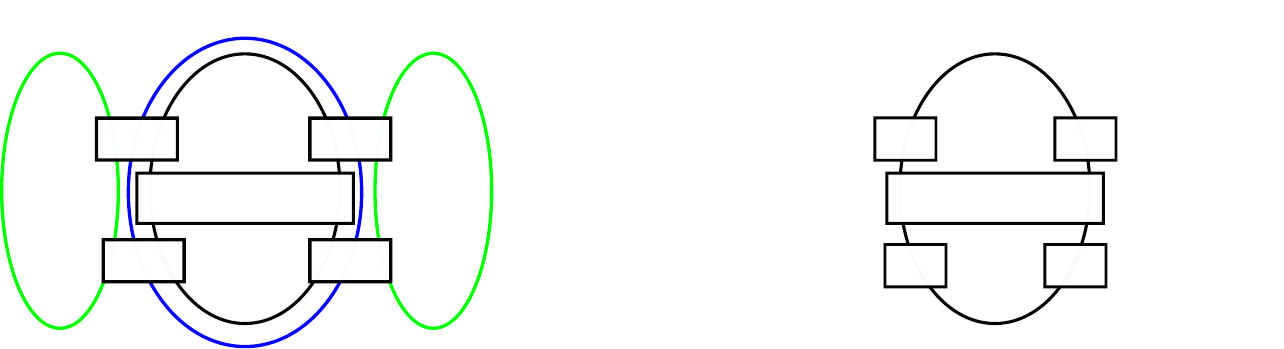}}%
    \put(0.15521435,0.25164778){\color[rgb]{0,0,0}\makebox(0,0)[lb]{\smash{$\ell$}}}%
    \put(0.31820854,0.25164778){\color[rgb]{0,0,0}\makebox(0,0)[lb]{\smash{$\ell'$}}}%
    \put(0.00649407,0.25164778){\color[rgb]{0,0,0}\makebox(0,0)[lb]{\smash{$\ell'$}}}%
    \put(0.15325963,0.10910007){\color[rgb]{0,0,0}\makebox(0,0)[lb]{\smash{$L^{n, j}$}}}%
    \put(0.17144298,0.18572987){\color[rgb]{0,0,0}\makebox(0,0)[lb]{\smash{$j$}}}%
    \put(0.73772426,0.10910007){\color[rgb]{0,0,0}\makebox(0,0)[lb]{\smash{$L^{n, j}$}}}%
    \put(0.39196937,0.11116937){\color[rgb]{0,0,0}\makebox(0,0)[lb]{\smash{$\doteq^s_{j-1}(-A)^{2(\ell+2\ell')}$}}}%
    \put(0.74811474,0.18572987){\color[rgb]{0,0,0}\makebox(0,0)[lb]{\smash{$j$}}}%
  \end{picture}%
\endgroup%

\end{equation}
where $s = |L^{n, j}|+ \deg \langle \overline{\Sk_+} \rangle.$  
\end{lem} 

\begin{proof}
We apply Lemma \ref{lem:avariant} to the pair of projectors on the right side of $\Sk$. 
\begin{equation}
\def \svgwidth{\columnwidth}
%% Creator: Inkscape inkscape 0.91, www.inkscape.org
%% Accompanies image file 'gformld.pdf' (pdf, eps, ps)
\begingroup%
  \makeatletter%
  \providecommand\color[2][]{%
    \errmessage{(Inkscape) Color is used for the text in Inkscape, but the package 'color.sty' is not loaded}%
    \renewcommand\color[2][]{}%
  }%
  \providecommand\transparent[1]{%
    \errmessage{(Inkscape) Transparency is used (non-zero) for the text in Inkscape, but the package 'transparent.sty' is not loaded}%
    \renewcommand\transparent[1]{}%
  }%
  \providecommand\rotatebox[2]{#2}%
  \ifx\svgwidth\undefined%
    \setlength{\unitlength}{961.03009081bp}%
    \ifx\svgscale\undefined%
      \relax%
    \else%
      \setlength{\unitlength}{\unitlength * \real{\svgscale}}%
    \fi%
  \else%
    \setlength{\unitlength}{\svgwidth}%
  \fi%
  \global\let\svgwidth\undefined%
  \global\let\svgscale\undefined%
  \makeatother%
  \begin{picture}(1,0.20524224)%
    \put(0,0){\includegraphics[width=\unitlength,page=1]{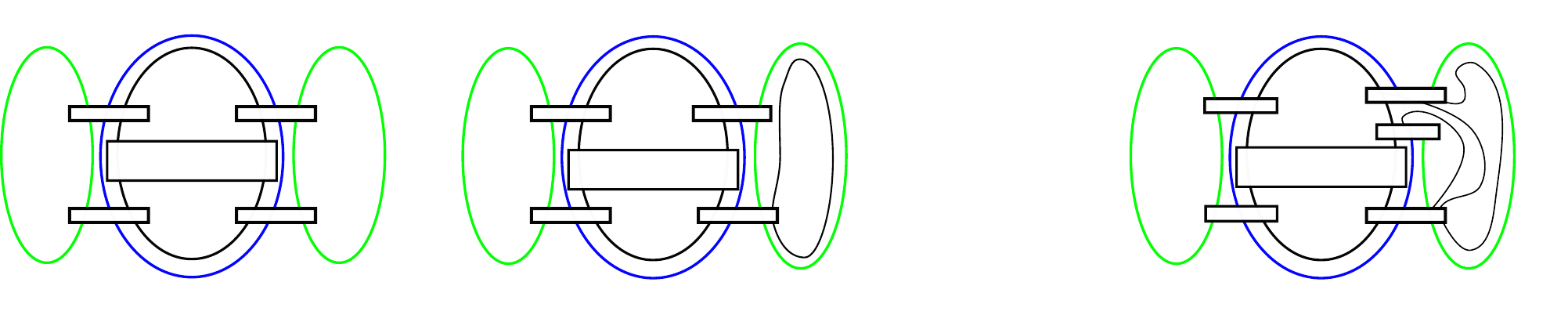}}%
    \put(0.10103688,0.09519758){\color[rgb]{0,0,0}\makebox(0,0)[lb]{\smash{$L^{n, j}$}}}%
    \put(0.38906115,0.09020293){\color[rgb]{0,0,0}\makebox(0,0)[lb]{\smash{$L^{n, j}$}}}%
    \put(0.49989456,0.08497045){\color[rgb]{0,0,0}\makebox(0,0)[lb]{\smash{$1$}}}%
    \put(0.0113712,0.1892633){\color[rgb]{0,0,0}\makebox(0,0)[lb]{\smash{$\ell'$}}}%
    \put(0.09473412,0.19259306){\color[rgb]{0,0,0}\makebox(0,0)[lb]{\smash{$\ell$}}}%
    \put(0.19795667,0.18902546){\color[rgb]{0,0,0}\makebox(0,0)[lb]{\smash{$\ell'$}}}%
    \put(0.30771986,0.18759842){\color[rgb]{0,0,0}\makebox(0,0)[lb]{\smash{$\ell'$}}}%
    \put(0.39108275,0.19092818){\color[rgb]{0,0,0}\makebox(0,0)[lb]{\smash{$\ell$}}}%
    \put(0.49430532,0.18736058){\color[rgb]{0,0,0}\makebox(0,0)[lb]{\smash{$\ell'-1$}}}%
    \put(0.25627267,0.10976526){\color[rgb]{0,0,0}\makebox(0,0)[lb]{\smash{\\ }}}%
    \put(0.26542951,0.10664709){\color[rgb]{0,0,0}\makebox(0,0)[lb]{\smash{$=$}}}%
    \put(0.10762266,0.15198187){\color[rgb]{0,0,0}\makebox(0,0)[lb]{\smash{$j$}}}%
    \put(0.40230642,0.15198187){\color[rgb]{0,0,0}\makebox(0,0)[lb]{\smash{$j$}}}%
    \put(0.70935789,0.24473947){\color[rgb]{0,0,0}\makebox(0,0)[lb]{\smash{}}}%
    \put(0.2273016,0.31736193){\color[rgb]{0,0,0}\makebox(0,0)[lt]{\begin{minipage}{0.27232678\unitlength}\raggedright \end{minipage}}}%
    \put(0.40956055,0.00392647){\color[rgb]{0,0,0}\makebox(0,0)[lb]{\smash{$\Sk^1$}}}%
    \put(0.09989281,0.00392647){\color[rgb]{0,0,0}\makebox(0,0)[lb]{\smash{$\Sk$}}}%
    \put(0.7311175,0.18550424){\color[rgb]{0,0,0}\makebox(0,0)[lb]{\smash{$\ell'$}}}%
    \put(0.81448045,0.188834){\color[rgb]{0,0,0}\makebox(0,0)[lb]{\smash{$\ell$}}}%
    \put(0.91770302,0.1852664){\color[rgb]{0,0,0}\makebox(0,0)[lb]{\smash{$\ell'-1$}}}%
    \put(0.82403918,0.14988771){\color[rgb]{0,0,0}\makebox(0,0)[lb]{\smash{$j$}}}%
    \put(0.82962838,0.00183227){\color[rgb]{0,0,0}\makebox(0,0)[lb]{\smash{$\Sk^2$}}}%
    \put(0.92816796,0.05933008){\color[rgb]{0,0,0}\makebox(0,0)[lb]{\smash{$1$}}}%
    \put(0.55183129,0.10529267){\color[rgb]{0,0,0}\makebox(0,0)[lb]{\smash{$+(-1)^{n-j}\frac{\triangle_{j-1}}{\triangle_{n-1}}$}}}%
    \put(0.81578856,0.08977364){\color[rgb]{0,0,0}\makebox(0,0)[lb]{\smash{$L^{n, j}$}}}%
  \end{picture}%
\endgroup%

\end{equation}
We write $\Sk^1$ for the first term of the sum. For the second term $\Sk^2$ of the sum, we have the local picture shown below where Lemma \ref{lem:avariant} was applied.
\begin{figure}[H]
\begin{center}\def\svgwidth{.4\columnwidth}
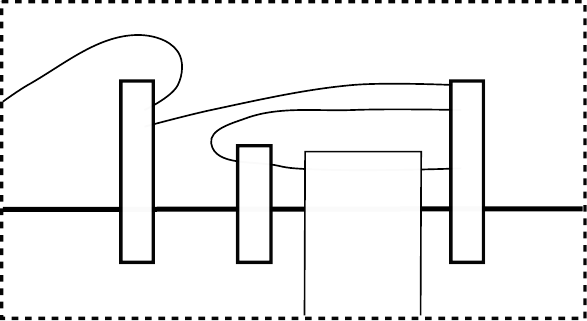
\end{center}

We compare the maximum degree of the Kauffman bracket of $\Sk^1$ and $\Sk^2$. Note first that $\deg \left( A^{\sgn(+)} \langle \overline{\Sk^1_+} \rangle \right) =  \deg \left(   A^{\sgn(+)} \langle  \overline{\Sk_+} \rangle \right) = s$, where $\sgn(+) = |L^{n, j}|$.  Now 
\[\deg \frac{\triangle_{j-1}}{\triangle_{n-1}}= \deg \frac{\triangle_{n-\ell-\ell'-1}}{\triangle_{n-1}} = -2(\ell+ \ell'). \] 
\end{figure}

Let $\sigma$ be a Kauffman state on the crossings of $\Sk^2$, then 
\[ \langle \Sk^2 \rangle = \sum_{\sigma} A^{\sgn(\sigma)}\langle\Sk^2_{\sigma} \rangle.  \]  
In order to estimate the maximum degree of $\langle \Sk^2\rangle$ relative to $A^{\sgn(+)}\langle \overline{ \Sk^1_+} \rangle$, we estimate the degree of each term in the sum above. 

Let $k$ be the largest number $\in \{1, \ldots, j\}$ such that $\sigma$ chooses the $-$-resolution on a crossing of $L^{n, j}$ between $S_{k-1}$ and $S_{k}$. The state $\sigma$ has to choose the $-$-resolution for some crossing between $S_{i-1}$ and $S_{i}$ for all $i \in \{2, \ldots, k \}$. Otherwise, there would be a cap or a cup composed with a projector that would make $\langle \Sk^{2'}_{\sigma} \rangle =0$. In the pictorial calculations to follow, we will denote by $L_{\sigma}$ the result of applying $\sigma$ to the crossings in $L^{n, j}$. Applying Lemma \ref{lem:avariant}, we have 
\begin{equation*}
\def \svgwidth{.9\columnwidth}
%% Creator: Inkscape inkscape 0.91, www.inkscape.org
%% Accompanies image file 'gflocalc3.pdf' (pdf, eps, ps)
\begingroup%
  \makeatletter%
  \providecommand\color[2][]{%
    \errmessage{(Inkscape) Color is used for the text in Inkscape, but the package 'color.sty' is not loaded}%
    \renewcommand\color[2][]{}%
  }%
  \providecommand\transparent[1]{%
    \errmessage{(Inkscape) Transparency is used (non-zero) for the text in Inkscape, but the package 'transparent.sty' is not loaded}%
    \renewcommand\transparent[1]{}%
  }%
  \providecommand\rotatebox[2]{#2}%
  \ifx\svgwidth\undefined%
    \setlength{\unitlength}{960.73051873bp}%
    \ifx\svgscale\undefined%
      \relax%
    \else%
      \setlength{\unitlength}{\unitlength * \real{\svgscale}}%
    \fi%
  \else%
    \setlength{\unitlength}{\svgwidth}%
  \fi%
  \global\let\svgwidth\undefined%
  \global\let\svgscale\undefined%
  \makeatother%
  \begin{picture}(1,0.39763444)%
    \put(0,0){\includegraphics[width=\unitlength,page=1]{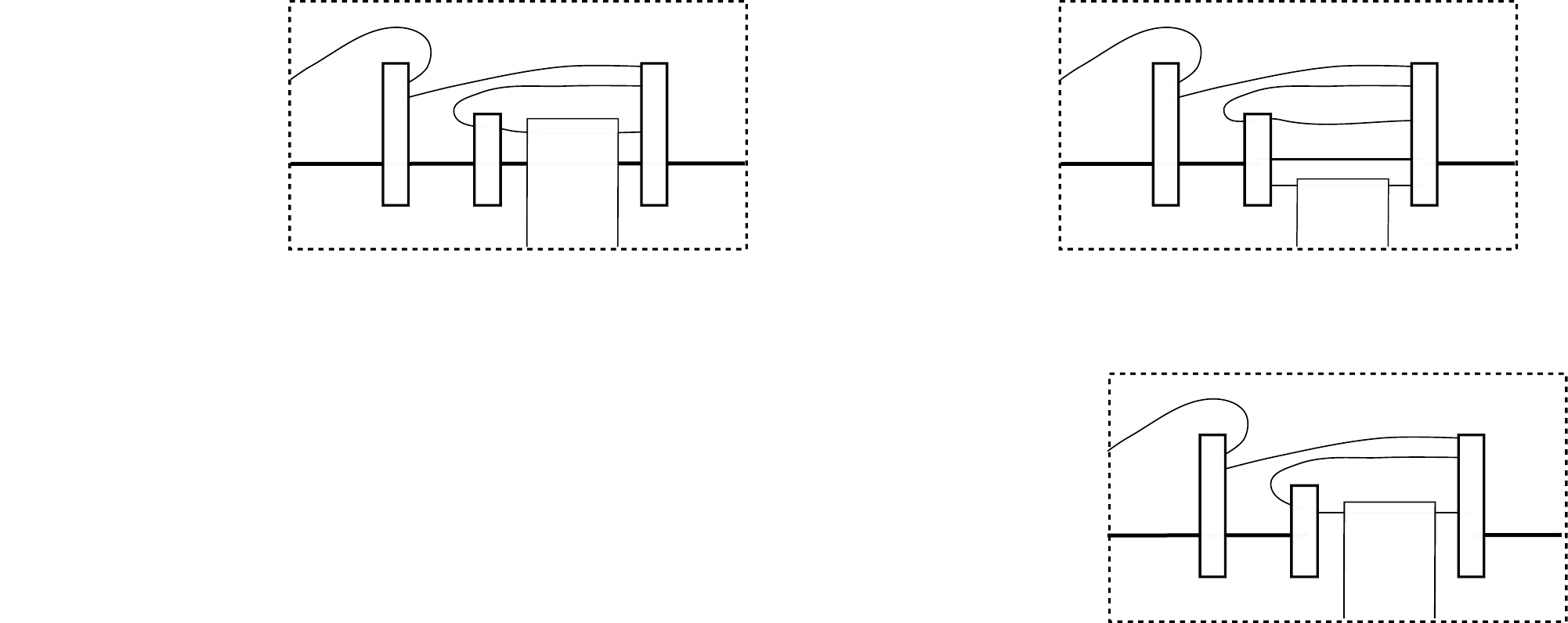}}%
    \put(0.33942975,0.27023139){\color[rgb]{0,0,0}\makebox(0,0)[lb]{\smash{$L_{\sigma}$}}}%
    \put(0.19370731,0.2410869){\color[rgb]{0,0,0}\makebox(0,0)[lb]{\smash{$n-1$}}}%
    \put(0.20203431,0.33851278){\color[rgb]{0,0,0}\makebox(0,0)[lb]{\smash{$1$}}}%
    \put(0.28030807,0.36099567){\color[rgb]{0,0,0}\makebox(0,0)[lb]{\smash{$\ell+\ell'-1$}}}%
    \put(0.36857424,0.32269149){\color[rgb]{0,0,0}\makebox(0,0)[lb]{\smash{$1$}}}%
    \put(0.26032329,0.2469158){\color[rgb]{0,0,0}\makebox(0,0)[lb]{\smash{$j-1$}}}%
    \put(0.4351902,0.3010413){\color[rgb]{0,0,0}\makebox(0,0)[lb]{\smash{$n$}}}%
    \put(0.90863938,0.29604509){\color[rgb]{0,0,0}\makebox(0,0)[lb]{\smash{}}}%
    \put(0.75875341,0.30437207){\color[rgb]{0,0,0}\makebox(0,0)[lb]{\smash{$1$}}}%
    \put(0.69380284,0.33768006){\color[rgb]{0,0,0}\makebox(0,0)[lb]{\smash{$1$}}}%
    \put(0.68547584,0.2410869){\color[rgb]{0,0,0}\makebox(0,0)[lb]{\smash{$n-1$}}}%
    \put(0.83369637,0.2494139){\color[rgb]{0,0,0}\makebox(0,0)[lb]{\smash{$L'_{\sigma}$}}}%
    \put(0.92695872,0.30437207){\color[rgb]{0,0,0}\makebox(0,0)[lb]{\smash{$n$}}}%
    \put(0.78706519,0.36099565){\color[rgb]{0,0,0}\makebox(0,0)[lb]{\smash{$\ell+\ell'-1$}}}%
    \put(0.75209181,0.2460831){\color[rgb]{0,0,0}\makebox(0,0)[lb]{\smash{$j-1$}}}%
    \put(0.84035797,0.29937587){\color[rgb]{0,0,0}\makebox(0,0)[lb]{\smash{$j-k$}}}%
    \put(0.49359958,0.07577567){\color[rgb]{0,0,0}\makebox(0,0)[lb]{\smash{$=(-1)^{n-k}\frac{\triangle_{k-1}}{\triangle_{n-1}}$}}}%
    \put(0.71367007,0.13715753){\color[rgb]{0,0,0}\makebox(0,0)[lb]{\smash{$1$}}}%
    \put(0.80193623,0.12549974){\color[rgb]{0,0,0}\makebox(0,0)[lb]{\smash{$n-k-1$}}}%
    \put(0.8602252,0.08386476){\color[rgb]{0,0,0}\makebox(0,0)[lb]{\smash{$1$}}}%
    \put(0.8602252,0.01225259){\color[rgb]{0,0,0}\makebox(0,0)[lb]{\smash{$L'_{\sigma}$}}}%
    \put(0.71533547,0.00392559){\color[rgb]{0,0,0}\makebox(0,0)[lb]{\smash{$n-1$}}}%
    \put(0.96014915,0.06721076){\color[rgb]{0,0,0}\makebox(0,0)[lb]{\smash{$n$}}}%
    \put(0.48812604,0.31507821){\color[rgb]{0,0,0}\makebox(0,0)[lb]{\smash{$=(-1)^{n-j}\frac{\triangle_{j-1}}{\triangle_{n-1}}$}}}%
    \put(-0.00138236,0.30913035){\color[rgb]{0,0,0}\makebox(0,0)[lb]{\smash{$(-1)^{n-j}\frac{\triangle_{j-1}}{\triangle_{n-1}}$}}}%
    \put(0.06761275,0.47091201){\color[rgb]{0,0,0}\makebox(0,0)[lb]{\smash{}}}%
    \put(0.78028614,0.00892179){\color[rgb]{0,0,0}\makebox(0,0)[lb]{\smash{$k-1$}}}%
    \put(0.85867741,0.29604508){\color[rgb]{0,0,0}\makebox(0,0)[lb]{\smash{}}}%
    \put(0.30076866,0.20444812){\color[rgb]{0,0,0}\makebox(0,0)[lb]{\smash{$\Sk^2_{\sigma}$}}}%
  \end{picture}%
\endgroup%
, 
\end{equation*}
where $L'_{\sigma}$ is the result of applying $\sigma$ to the set of crossings of $L^{n, j}$ inside of $S_k$. Now we estimate $\sgn(\sigma) + 2|\overline{\Sk_{\sigma}^{2'}}|$ relative to $\sgn(+)+2|\overline{\Sk^{2'}_{+}}|$.  The condition that $\sigma$ has to choose the $-$-resolution for a crossing between $S_{i-1}$ and $S_{i}$ for all $i\in \{2, \ldots, k\}$ forces $\deg \left( A^{\sgn(\sigma)} \langle \overline{\Sk^{2'}_{\sigma}} \rangle \right) \leq \deg \left( A^{\sgn(+)} \langle \overline{\Sk_+^{2'}} \rangle \right)- 4(k-1)$ by Lemma \ref{lem:sstates}. This is because each change of resolution from $+$- to $-$- merges $S_i$ and $S_{i+1}$. Now the fact that $|\overline{\Sk^{2'}_+}| = |\overline{\Sk^{1}_+}|-(n-k)$ and $\deg\frac{\triangle_{k-1}}{\triangle_{n-1}} = -2(n-k)$ then gives 
\[ \deg \left( A^{\sgn(\sigma)} \langle \overline{\Sk^{2'}_{\sigma}}\rangle \right) \leq \deg \left( A^{\sgn(+)} \langle \overline{\Sk^1_+}\rangle \right) -4(n-1). \]
Thus \[\Sk \doteq^s_{n-1} \Sk^1.\] At this point either we are dealing with a circle in the $\ell$ circles and we can apply Lemma \ref{lem:avariant} again to the bottom pair of projectors, or, if we are dealing with a circle in a pair of $\ell'$ pairs of circles, we simply have a loop attached to a Jones-Wenzl projector that we can pull off via an $\doteq^s_{n-1}$ equivalence using \eqref{eq:jwloopid} by rewriting 
\begin{align*}\frac{\triangle_{n+1}}{\triangle_{n}} &= -\frac{A^{2(n+2)}-A^{-2(n+2)}}{A^{2(n+1)}-A^{-2(n+1)}} =  -A^2 \frac{1-A^{-4(n+2)}}{1-A^{-4(n+1)}}\\
&=-A^2(1+ A^{-4(n+1)} + \cdots + \text{ terms with powers lower than $-4n$}).  \end{align*}See Figure \ref{fig:gformg1} for an illustration of both of these cases. 

\begin{figure}[H]
\def \svgwidth{.7\columnwidth}
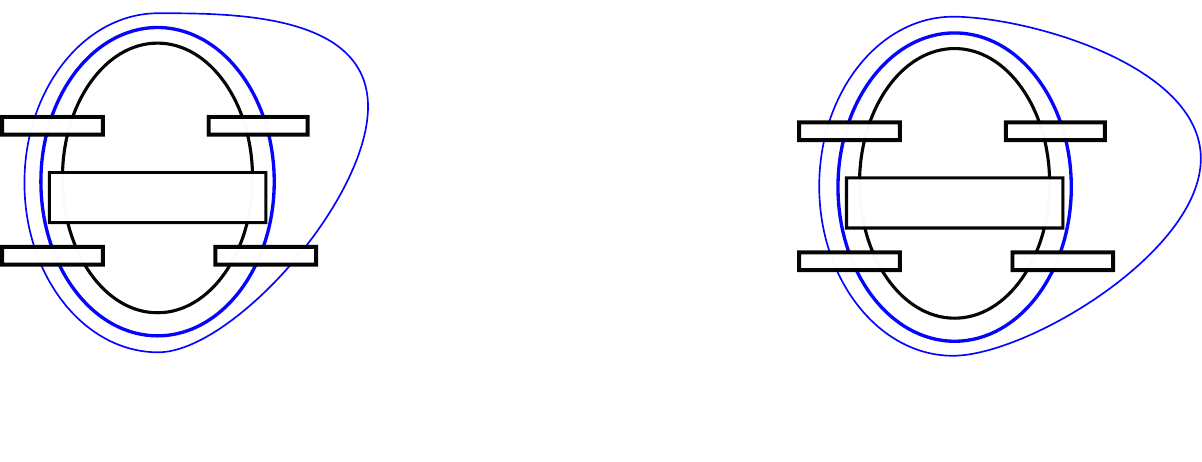
\caption{\label{fig:gformg2}The case when $\ell=0$.}
\end{figure} 

\begin{figure}[H]
\def \svgwidth{.7\columnwidth}
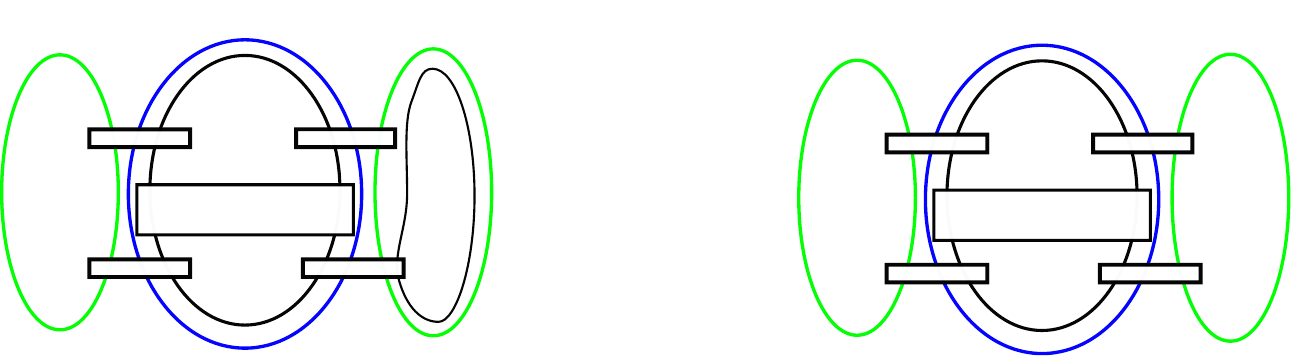
\caption{\label{fig:gformg1}The case when $\ell>0$.}
\end{figure} 

We repeat this process to detach each of the $2\ell' + \ell$ circles, keeping track of the change in the $\doteq$-equivalences when the number of strands through each projector decreases. This finishes the proof of the lemma. 
  
\end{proof}

We now prove Theorem \ref{thm:equiv}. Recall that
\begin{equation} \langle D^n_{\jwproj} \rangle = \sum_{\sigma} A^{\sgn(\sigma)} \langle \mathcal{S}^n_{\sigma} \rangle. \end{equation}
Since $H_n(D) = \deg \left( A^{\sgn(+)} \langle \overline{\Sk^n_+} \rangle \right)$, it suffices to show that 
\[\deg{\left( \sum_{\sigma} A^{\sgn(\sigma)}\langle \Sk^n_{\sigma} \rangle \right)} \leq  \deg \left( A^{\sgn(+)}\langle \overline{\Sk^n_+} \rangle \right) - 4(n-1). \] 

This follows from the technical lemma below. 

\begin{lem} \label{lem:main} Let $D$ be a non $+$-adequate diagram with a crossing $c$ corresponding to a loop inside a state circle in its all-$+$ state graph. 
Let $\Sk^n_{\sigma}$ be a skein obtained by applying a Kauffman state $\sigma$ to the crossings of $D^n_{\vcenter{\hbox{\includegraphics[scale=.1]{jwproj.png}}}}$, which is decorated with four Jones-Wenzl projectors framing $c$. Consider $\sigma$ such that 

\begin{equation} \deg \left( A^{\sgn(\sigma)} \langle \overline{\Sk_{\sigma}^n} \rangle \right) > \deg \left( A^{\sgn(+)} \langle \overline{\Sk^n_+} \rangle\right) - 4(n-1). \label{eq:contributingskein} \end{equation}

We have 
\[\deg{\left( \sum_{\sigma \text{ satisfying } \eqref{eq:contributingskein}} A^{\sgn(\sigma)} \langle \Sk_{\sigma}^n \rangle \right)} \leq \deg \left(  A^{\sgn(+)} \langle \overline{\Sk^n_+} \rangle \right) - 4(n-1). \]
\end{lem}

By disregarding the skeins from states whose Kauffman brackets have maximum degrees which are too low, we have
\begin{align} \langle D^n_{\jwproj} \rangle = \sum_{\sigma} A^{\sgn(\sigma)} \langle S_{\sigma}^n \rangle \doteq_{n-1}^s \left( \sum_{\sigma \text{ satisfying } \eqref{eq:contributingskein}} A^{\sgn(\sigma)} \langle \Sk_{\sigma}^n \rangle \right). 
\end{align}
Applying Lemma \ref{lem:main} then completes the proof of Theorem \ref{thm:equiv}.

\begin{proof}(of Lemma \ref{lem:main}) \\ 
A Kauffman state $\sigma$ satisfying \eqref{eq:contributingskein} chooses the $+$-resolution on all crossings whose corresponding segments lie between $S_{i}$, $S_{i+1}$ in $\overline{\Sk^n_+}$ for some $1\leq i \leq n-1$. To see this, we compare $\sigma$ to the all-$+$ state. If $\sigma$ chooses the $-$-resolution at a crossing whose corresponding segment lies between $S_i$ and $S_{i+1}$ for every $1\leq i \leq n-1$, then 
\[\deg \left( A^{\sgn(\sigma)} \langle \overline{\Sk_{\sigma}^n} \rangle \right) \leq \deg \left(  A^{\sgn(+)}\langle \overline{\Sk^n_+} \rangle \right) - 4(n-1) \] by taking a sequence of states from the all-$+$ state to $\sigma$ where there are $n-1$ pairs of terms in the sequence, each of which merges a pair of circles, and applying Lemma \ref{lem:sstates}.  

For a Kauffman state $\sigma$  on $D^n_{\jwproj}$, let $j_{\sigma}$ be the largest integer in $\{1, \ldots, n-1 \}$ where $\sigma$ chooses the $+$-resolution for all the crossings corresponding to segments between $S_i$ and $S_{i+1}$. In $D^n_{\jwproj}$, the loop crossing $c$ cables to $c^n$ crossings. Let $L_j$ be the subset of crossings of $c^n$ whose corresponding segments in $\Sk^n_{+}$ lie inside $S_j$. We define an equivalence relation on the set of skeins with Kauffman brackets satisfying \eqref{eq:contributingskein}: Two skeins $\Sk^n_{\sigma}$ and $\Sk^n_{\sigma'}$ resulting from applying the Kauffman states $\sigma$ and $\sigma'$ to $D^n_{\jwproj}$, respectively, are equivalent, and we write $\Sk^n_{\sigma} \sim \Sk^n_{\sigma'}$, if and only if 
\begin{enumerate}[(a)]
\item $j_{\sigma}=j_{\sigma'}$.  
\item $\sigma$ and $\sigma'$ are identical outside of $L_{j_{\sigma}}$. 
\end{enumerate}
It is clear that $\sim$ is an equivalence relation. In an equivalence class of $\sigma$, we may decompose $\sigma$ as a disjoint union of Kauffman states $\sigma_1 \sqcup \sigma_2$, where $\sigma_1$ is on the crossings in $L_{j_{\sigma}}$, and $\sigma_2$ is on the crossings not in $L_{j_{\sigma}}$. So we have $\sgn(\sigma) = \sgn(\sigma_1) + \sgn(\sigma_2)$.

We have 
\begin{align} \label{eq:starts}
\sum_{\sigma \text{ satisfying } \eqref{eq:contributingskein}} A^{\sgn(\sigma)} \langle \Sk^{n}_{\sigma} \rangle &= 
\sum_{C \text{ an equivalence class of } \sim} \ \ \sum_{\sigma \in C} A^{\sgn(\sigma)} \langle \Sk^{n}_{\sigma} \rangle. \\ \notag
\intertext{Since $\sgn(\sigma_2)$ is identical across $C$, this is equal to} 
\sum_{\sigma \text{ satisfying } \eqref{eq:contributingskein}} A^{\sgn(\sigma)} \langle \Sk^{n}_{\sigma} \rangle &= \sum_{C \text{ an equivalence class of } \sim} \ \ A^{\sgn(\sigma_2)}\sum_{\sigma \in C} A^{\sgn(\sigma_1)} \langle \Sk^{n}_{\sigma} \rangle. 
\intertext{ Fix an equivalence class $C$ and let $\Sk^{n, C}$ be the skein resulting from applying $\sigma_2$ to the crossings not in $L_{j_{\sigma}}$ for $\sigma \in C$. Then}
\sum_{\sigma \text{ satisfying } \eqref{eq:contributingskein}} A^{\sgn(\sigma)} \langle \Sk^{n}_{\sigma} \rangle&=  \sum_{C \text{ an equivalence class of } \sim} \ \ A^{\sgn(\sigma_2)} \langle \Sk^{n, C}\rangle.
\end{align}

 The component $\Sk$ decorated by projectors in $\Sk^{n, C}$ may be isotoped to have the form of Figure \ref{fig:twocase} with $ n-(\ell+\ell') = j_{\sigma}$ for $\sigma\in C$, since otherwise there is a cap or a cup composed with a projector, resulting in a 0 skein. Let $\ell'$ be the number of pairs of circles, each going through two projectors and $\ell$ be the number of circles through all four projectors outside of $S_{j_{\sigma}}$.  Let $s' = |L_{j_{\sigma}}| + \deg \langle \overline{\Sk^{n, C}_+}\rangle$, where $\Sk_+^{n, C}$ is the skein resulting from applying the all-$+$ state to the crossings of $\Sk^{n, C}$, and let $|\sigma_2|$ be the number of circles disjoint from the component decorated by the projectors in $\Sk^{n, C}$. Now it follows directly from Lemma \ref{lem:localform} that
\begin{equation} \label{eq:doteq} \langle \Sk^{n, C}\rangle  \doteq_{n-1}^{s'} (-A^{-2}-A^2)^{|\sigma_2|} (-A)^{2(\ell+2\ell')} \langle U^{n, j_{\sigma}}\rangle, \end{equation}
where $U^{n,j_{\sigma}}$ is the skein as defined by Definition \ref{defn:unknotwist}.

By Lemma \ref{lem:generic-loop}, 
\[\deg \langle U^{n, j_{\sigma}}  \rangle \leq |L_{j_{\sigma}}| + \deg \langle \overline{U^{n, j_{\sigma}}_+} \rangle -4j_{\sigma}. \] 
This implies 
\begin{equation} \deg \langle \Sk^{n, C} \rangle \leq  s' - 4j_{\sigma}.  \label{eq:degestimatesnc}  \end{equation}  
Take a sequence of states from the all-$+$ state to $\sigma = \sigma_2 \sqcup \sigma_1$. Since $\sigma_2$ chooses the $-$-resolution on a crossing between $S_i$ and $S_{i+1}$ for every $i\in \{j_{\sigma}, \ldots, n-1\}$, we have by applying Lemma \ref{lem:sstates} that 
\[s'+ \sgn(\sigma_2) + 2|\sigma_2|\leq \sgn(+) + \deg \langle \overline{\Sk^n_+} \rangle - 4(n-j_{\sigma}-1).  \]  Taken with \eqref{eq:degestimatesnc}, we get 
\[\deg \left( \sum_{C \text{ an equivalence class of } \sim}  A^{\sgn(\sigma)} \langle \Sk^n_{\sigma} \rangle \right) \leq \deg \left( A^{\sgn(+)} \langle \overline{\Sk^n_+}\rangle \right) - 4(n-1).   \]

\end{proof} 

\section{Detecting semi-adequacy using the colored Jones polynomial} \label{sec:invariants} 
We use Theorem~\ref{thm:tail} to define a link invariant as in \cite{kalee1}. Let $D$ be a diagram of an oriented link $K$. Let $c_-(D)$ be the number of negative crossings of $D$, and recall that $|s_+(D)|$ is the number of state circles in the all $+$-resolution of $D$, $c(D)$ is the number of crossings in $D$, and $\omega(D)$ is the writhe of $D$. We consider the complexity 
\[ (c_-(D), c(D), |s_+(D)| - \omega(D)),\]  
ordered lexicographically. Let $\mathcal{D}(K)$ be the set of diagrams of $K$ which minimizes this complexity. Recall that the lower bound $h_n(D)$ of the minimum degree $d(n)$ of $J_K^n(q)$, defined in \eqref{eq:lowerbound},  is $-\frac{1}{4}H_{n-1}(D) + \frac{(n-1)^2+2(n-1)}{4} \omega(D)$, where 
\[H_n(D) = n^2c(D) + 2n|s_+(D)|. \] 
\begin{defn} Let $K$ be a link and $D$ an oriented link diagram in $\mathcal{D}(K)$. For $i\geq 3$, let
$\beta_i=\beta_i(D)$ be the coefficient of $q^{h_i(D)+i-3}$ in $J_K^i(q)$.
Define
$$J^+_D(q):=\sum_{i=3}^{\infty} \beta_i q^{i}.$$
\end{defn} 
We will need the following lemma from \cite{kalee1}. 
\begin{lem}\label{lem:alladequate} \cite[Lemma 3.4]{kalee1} 
Suppose that for a link $K$, there is a diagram $D\in \mathcal{D}(K)$ that is $+$-adequate. Then, all the diagrams in $\mathcal{D}(K)$ are $+$-adequate. 
\end{lem} 
Applying Theorem ~\ref{thm:tail}, we have the following corollaries. 
\begin{cor} 
$J^+_D(q) \not= 0$ if and only if $D$ is $+$-adequate. 
\end{cor} 
\begin{proof}
If $D$ is not $+$-adequate, then Theorem ~\ref{thm:tail} says that $d(i) \geq h_{i}(D) + i-2$, so $\beta_i = 0$ for all $i$. This shows the forward direction. For the converse, $\beta_3$ is the coefficient of $q^{h_3(D)}$ in $J_K^3(q)$. If $D$ is $+$-adequate, then $h_3(D)$ is equal to the minimum degree $d(3)$ of $J_K^3(q)$, so $\beta_3 \not=0$, and this shows that $J_D^+(q) \not= 0$. 
\end{proof} 
\begin{cor} The power series $J^+_D(q)$ defined above is independent of the diagram $D \in \mathcal{D}(K)$, thus it is an invariant of $K$, which we denote by $J^+_K(q)$. 
\end{cor} 
\begin{proof} If $K$ is not $+$-adequate, then any diagram in $\mathcal{D}(K)$ is not $+$-adequate. Let $D$ be a diagram in  $\mathcal{D}(K)$, then by Theorem ~\ref{thm:tail}, $J^+_D(q) = 0$. If $K$ is $+$-adequate, then an $+$-adequate diagram of $K$ minimizes the complexity $(c_-(D), c(D), |s_+(D)| - w(D))$, thus it belongs to $\mathcal{D}(K)$,  and all the diagrams in $\mathcal{D}(K)$ are $+$-adequate by Lemma ~\ref{lem:alladequate}. Let $D$ be a diagram in $\mathcal{D}(K)$. As shown in \cite{Arm13},  $J^+_D(q)$ records the stable coefficients of the sequence $\{J_K^n(q)\}_{n=2}^{\infty}$, therefore it is also independent of the diagram $D$. 
\end{proof} 

\begin{defn} Let $D$ be a link diagram. Consider the graph $s_+(D)$ with vertices the state circles, and edges the segments from the all-$+$ Kauffman state of $D$.  We denote by $\chi_+(D)$ the Euler characteristic of $s_+(D)$.  
\end{defn} 
\begin{cor} Suppose $D$ is an $+$-adequate diagram of a link $K$ and $D'$ is another diagram of $K$. Then $D'$ is $+$-adequate if and only if $c_-(D) = c_-(D')$ and $\chi_+(D) = \chi_+(D')$.  
\end{cor} 
\begin{proof}
If $D'$ is $+$-adequate, then $c_-(D)$ and $|s_+(D)| - w(D)$ are invariants of $K$ \cite[Theorem 5.13]{Lic97}.  Thus, $\chi_+(D) = |s_+(D)| - c(D)= |s_+(D)| - w(D) - 2c_-(D)$ is also an invariant of $K$. For the converse, since $D$ is $+$-adequate, the minimum degree $d(n+1)$ of $J_K^{n+1}(q)$ is equal to $h_{n+1}(D)$ for all $n\geq 1$. We rewrite $h_{n+1}(D)$ here slightly differently: 
\begin{align*}
h_{n+1}(D) &= -\frac{1}{4}(n^2c(D) + 2n|s_+(D)| -\omega(D)(n^2+2n)) \\ 
&=-\frac{1}{4}(2c_-(D)n^2 + 2n(|s_+(D)|-\omega(D)))
\end{align*}
If $D'$ is not $+$-adequate but $c_-(D) = c_-(D')$ and $\chi_+(D) = \chi_+(D')$, then $|s_+(D)| - \omega(D) = |s_+(D')| - \omega(D')$, so $h_{n+1}(D) = h_{n+1}(D')$. Theorem \ref{thm:tail} applied to $D'$ will imply that $d(n+1) < h_{n+1}(D)$ for $n\geq 2$, which is a contradiction.
 \end{proof}

\bibliographystyle{amsalpha}
\bibliography{references}

\newcommand{\etalchar}[1]{$^{#1}$}
\providecommand{\bysame}{\leavevmode\hbox to3em{\hrulefill}\thinspace}
\providecommand{\MR}{\relax\ifhmode\unskip\space\fi MR }
\providecommand{\MRhref}[2]{%
  \href{http://www.ams.org/mathscinet-getitem?mr=#1}{#2}
}
\providecommand{\href}[2]{#2}
\begin{thebibliography}{MMO{\etalchar{+}}02}

\bibitem[AD]{codyoliver}
Cody Armond and Oliver~T. Dasbach, \emph{Rogers--{R}amanujan type identities
  and the head and tail of the colored {J}ones polynomial}, arXiv:1106.3948.

\bibitem[AD17]{AD13}
\bysame, \emph{The head and tail of the colored {J}ones polynomial for adequate
  knots}, Proceedings of the American Mathematical Society \textbf{145} (2017),
  no.~3, 1357--1367.

\bibitem[Arm13]{Arm13}
Cody Armond, \emph{The head and tail conjecture for alternating knots.},
  Algebr. Geom. Topol. \textbf{13} (2013), no.~5, 2809--2826.

\bibitem[CK12]{CK12}
Benjamin Cooper and Vyacheslav Krushkal, \emph{Categorification of the
  {J}ones-{W}enzl projectors}, Quantum Topology \textbf{3} (2012), no.~2,
  139--180.

\bibitem[DL06]{DL06}
Oliver~T. Dasbach and Xiao-Song Lin, \emph{On the head and the tail of the
  colored {J}ones polynomial}, Compositio Mathematica \textbf{142} (2006),
  no.~5, 1332--1342.

\bibitem[DL07]{DL07}
\bysame, \emph{A volum-ish theorem for the {J}ones polynomial of alternating
  knots}, Pacific J. Math. \textbf{231} (2007), no.~2, 279--291.

\bibitem[FGL02]{FGL02}
Charles Frohman, Razvan Gelca, and Walter Lofaro, \emph{The {A}-polynomial from
  the noncommutative viewpoint}, Trans. Amer. Math. Soc. \textbf{354} (2002),
  no.~2, 735--747.

\bibitem[FKP08]{fkp:filling}
David Futer, Efstratia Kalfagianni, and Jessica~S. Purcell, \emph{{Dehn
  filling, volume, and the {J}ones polynomial}}, J. Differential Geom.
  \textbf{78} (2008), no.~3, 429--464.

\bibitem[FKP11]{fkp:slope}
\bysame, \emph{Slopes and colored {J}ones polynomials of adequate knots}, Proc.
  Amer. Math. Soc. \textbf{139} (2011), no.~5, 1889--1896.

\bibitem[FKP13]{fkp:gutsjp}
\bysame, \emph{Guts of surfaces and the colored {J}ones polynomial}, Lecture
  Notes in Mathematics, vol. 2069, Springer, Heidelberg, 2013.

\bibitem[FKP14]{fkp:qsf}
\bysame, \emph{Quasifuchsian state surfaces}, Trans. Amer. Math. Soc.
  \textbf{366} (2014), 4323--4343.

\bibitem[FKS06]{FKS06}
Igor Frenkel, Mikhail Khovanov, and Catharina Stroppel, \emph{A
  categorification of finite-dimensional irreducible representations of quantum
  {$\mathfrak{sl}_2$} and their tensor products}, Selecta Mathematica. New
  Series \textbf{12} (2006), no.~3-4, 379--431.

\bibitem[Gar11]{garoufalidis}
Stavros Garoufalidis, \emph{The {J}ones slopes of a knot}, Quantum Topology
  \textbf{2} (2011), 43--69.

\bibitem[GL05]{q-holo}
Stavros Garoufalidis and Thang T.~Q. L{\^e}, \emph{The colored {J}ones function
  is q-holonomic}, Geom. Topology. (2005), no.~9, 1253--1293.

\bibitem[GL15]{garoufalidisLe}
\bysame, \emph{Nahm sums, stability and the colored {J}ones polynomial},
  Research in Mathematical Sciences \textbf{2} (2015), 1--55.

\bibitem[GNV16]{garouvong1}
Stavros Garoufalidis, Sergei Norin, and Thao Vong, \emph{Flag algebras and the
  stable coefficients of the {J}ones polynomial}, European J. Combinatorics
  \textbf{51} (2016), 156--189.

\bibitem[GV17]{garouvong}
Stavros Garoufalidis and Thao Vong, \emph{A stability conjecture for the
  colored {J}ones polynomial}, Topology Proceedings \textbf{49} (2017),
  211--249.

\bibitem[JCC14]{knotinfo}
C.~Livingston J.~C.~Cha, \emph{Knotinfo: Table of knot invariants}, 2014.

\bibitem[Kal18]{Kal16}
Efstratia Kalfagianni, \emph{A {J}ones slopes characterization of adequate
  knots}, Indiana Univ. Mathematics Journal \textbf{67} (2018), no.~1,
  205--219.

\bibitem[Kas97]{Kas97}
R.~M. Kashaev, \emph{The hyperbolic volume of knots from the quantum
  dilogarithm}, Lett. Math. Phys. \textbf{39} (1997), no.~3, 269--275.

\bibitem[Kau87]{kauffman}
Louis~H. Kauffman, \emph{State models and the {J}ones polynomial}, Topology
  \textbf{26} (1987), no.~3, 395--407.

\bibitem[KL14]{kalee1}
Efstratia Kalfagianni and Christine Ruey~Shan Lee, \emph{On the degree of the
  colored {J}ones polynomial}, Acta Mathematica Vietnamica (Proceedings of
  Quantum Topology and Hyperbolic Geometry in Nha Trang, May 2013) \textbf{39}
  (2014), no.~4, 549--560.

\bibitem[Lee18]{Lee17}
Christine Ruey~Shan Lee, \emph{A trivial tail homology for non-{$A$}-adequate
  links}, Algebr. Geom. Topol. \textbf{18} (2018), no.~3, 1481--1513.

\bibitem[Lic97]{Lic97}
William B.~R. Lickorish, \emph{An introduction to knot theory}, Graduate Texts
  in Mathematics, vol. 175, Springer-Verlag, New York, 1997.

\bibitem[LT88]{lick-thistle}
W.~B.~Raymond Lickorish and Morwen~B. Thistlethwaite, \emph{Some links with
  nontrivial polynomials and their crossing-numbers}, Comment. Math. Helv.
  \textbf{63} (1988), no.~4, 527--539.

\bibitem[Man04]{manchon}
P.~M.~G. Manch{\'o}n, \emph{Extreme coefficients of {J}ones polynomials and
  graph theory}, J. Knot Theory Ramifications \textbf{13} (2004), no.~2,
  277--295.

\bibitem[MM01]{MM01}
Hitoshi Murakami and Jun Murakami, \emph{The colored {J}ones polynomials and
  the simplicial volume of a knot}, Acta Math. \textbf{186} (2001), no.~1,
  85--104.

\bibitem[MMO{\etalchar{+}}02]{MM02}
Hitoshi Murakami, Jun Murakami, Miyuki Okamoto, Toshie Takata, and Yoshiyuki
  Yokota, \emph{Kashaev's {C}onjecture and the {C}hern-{S}imons {I}nvariants of
  {K}nots and {L}inks}, Experiment. Math. \textbf{11} (2002), no.~3, 427--435.

\bibitem[Roz14a]{Roz10}
Lev Rozansky, \emph{An infinite torus braid yields a categorified
  {J}ones-{W}enzl projector}, Fundamenta Mathematica \textbf{225} (2014),
  no.~1, 305--326.

\bibitem[Roz14b]{Roz12}
\bysame, \emph{Khovanov homology of a unicolored {B}-adequate link has a tail},
  Quantum Topology \textbf{5} (2014), no.~4, 541--579.

\bibitem[Sto11]{stoimenow:coeffs}
Alexander Stoimenow, \emph{Coefficients and non-triviality of the {J}ones
  polynomial}, Journal f$\ddot{\text{u}}$r die Reine und Angewandte Mathematik
  \textbf{657} (2011), 1--55.

\bibitem[Wen87]{Wen87}
Hans Wenzl, \emph{On sequences of projections}, C.R. Math. Rep. Acad. Sci.
  Canada \textbf{9} (1987), no.~1, 5--9.

\end{thebibliography}

\end{document}